\documentclass{amsart}
\usepackage{amsmath,amssymb,amsthm}
\usepackage{mathtools,scalerel}
\usepackage{bbm,pifont}
\usepackage{enumerate}
\usepackage{xspace,cmap}

\newtheorem{thm}{Theorem}[section]
\newtheorem{lemma}[thm]{Lemma}
\newtheorem{prop}[thm]{Proposition}
\newtheorem{cor}{Corollary}
\newtheorem{coro}[thm]{Corollary}
\newtheorem{claim}[thm]{Claim}
\newtheorem{subclaim}{Claim}[thm]
\newtheorem{fact}[thm]{Fact}
\newtheorem*{thma}{Theorem A}
\newtheorem*{thmb}{Theorem B}
\newtheorem*{thmbp}{Theorem B'}
\newtheorem*{thmbpp}{Theorem B''}
\newtheorem*{thmc}{Theorem C}
\theoremstyle{definition}
\newtheorem{defn}[thm]{Definition}
\newtheorem{notation}[thm]{Notation}

\newcommand{\vmark}{\ding{51}}
\newcommand{\xmark}{\ding{55}}
\newcommand\s{\subseteq}
\newcommand\sq{\sqsubseteq}
\newcommand\rest{\mathbin{\mathrlap{\vstretch{0.75}{\restriction}}\restriction}}
\newcommand\fa[1]{\textup{SDFA}(\mathcal P_{#1})}
\newcommand\bd{\textup{bd}}
\newcommand\fin{\textup{fin}}
\newcommand\bb{\mathbb}
\newcommand\one{\mathbbm{1}_\mathbb{P}}
\newcommand\br{\blacktriangleright}
\newcommand\coi{c.o.i.\@\xspace}
\newcommand\cois{c.o.i.'s\xspace}
\newcommand\axiomfont[1]{\textsf{\textup{#1}}}
\newcommand\zfc{\axiomfont{ZFC}}
\newcommand\gch{\axiomfont{GCH}}
\newcommand\sh{\axiomfont{SH}}
\newcommand\tp{\axiomfont{TP}}
\newcommand\ch{\axiomfont{CH}}
\newcommand\Mid{\mathrel{\bigg|}\allowbreak}
\renewcommand\mid{\mathrel{|}\allowbreak}
\renewcommand\restriction{\mathbin\upharpoonright}
\renewcommand\eqref[1]{(\ref{#1})}

\DeclareMathOperator\im{Im}
\DeclareMathOperator\dbl{double}
\DeclareMathOperator\lht{Lev}
\DeclareMathOperator\cf{cf}
\DeclareMathOperator\dom{dom}
\DeclareMathOperator\otp{otp}
\DeclareMathOperator\add{Add}
\DeclareMathOperator\nacc{nacc}
\DeclareMathOperator\acc{acc}
\DeclareMathOperator\pred{pred}
\DeclareMathOperator\height{ht}
\DeclareMathOperator\ssup{ssup}
\DeclareMathOperator\cl{cl}


\author{Chris Lambie-Hanson}
\address{Department of Mathematics, Bar-Ilan University, Ramat-Gan 5290002, Israel.}
\urladdr{http://math.biu.ac.il/\textasciitilde lambiec}

\author{Assaf Rinot}
\address{Department of Mathematics, Bar-Ilan University, Ramat-Gan 5290002, Israel.}
\urladdr{http://www.assafrinot.com}

\title{A forcing axiom deciding the generalized {S}ouslin {H}ypothesis}
\thanks{This research was partially supported by the Israel Science Foundation (grant $\#$1630/14).}

\begin{document}
\begin{abstract} We derive a forcing axiom from the conjunction of square and diamond, and present a few applications,
primary among them being the existence of super-Souslin trees.
It follows that for every uncountable cardinal $\lambda$, if
$\lambda^{++}$ is not a Mahlo cardinal in G\"odel's constructible universe,
then $2^\lambda = \lambda^+$ entails the existence of a  $\lambda^+$-complete $\lambda^{++}$-Souslin tree.
\end{abstract}
\maketitle

\section{Introduction}
\subsection{Trees}
A \emph{tree} is a partially ordered set $(T,<_T)$ with the property that for every $x\in T$,
the downward cone $x_\downarrow:=\{ y\in T\mid y<_T x\}$ is well-ordered by $<_T$. The order type of $(x_\downarrow,<_T)$ is denoted by $\height(x)$,
and the \emph{$\alpha^{th}$-level} of the tree is the set $T_\alpha:=\{x\in T\mid \height(x)=\alpha\}$.
The tree $(T,<_T)$ is said to be \emph{$\chi$-complete} if for every chain $C\s T$ of size $<\chi$, there is $x\in T$ such that $C\s x_\downarrow\cup\{x\}$.

If $\kappa$ is a regular, uncountable cardinal, then
a \emph{$\kappa$-Aronszajn tree} is a tree of size $\kappa$ having no chains or levels of size $\kappa$, and
a \emph{$\kappa$-Souslin tree} is a tree of size $\kappa$ having no chains or antichains of size $\kappa$.
As tree levels are antichains, any $\kappa$-Souslin tree is a $\kappa$-Aronszajn tree.

In 1920, Mikhail Souslin \cite{Souslin} asked whether every ccc, dense, complete linear ordering with no endpoints is isomorphic to
the real line.\footnote{Here, \emph{ccc} is a consequence of separability, asserting that every pairwise-disjoint family of open intervals is countable.}
In \cite{kurepa1935ensembles}, Kurepa showed that a negative answer to Souslin's question is equivalent to the existence of an
$\aleph_1$-Souslin tree. Attempts to settle the question by constructing an $\aleph_1$-Souslin tree proved unsuccessful but did
lead to Aronszajn's construction of an $\aleph_1$-Aronszajn tree, which is described in \cite{kurepa1935ensembles}.
The question remained open until it was proven, in \cite{MR0224456}, \cite{MR0215729}, \cite{jensen1968souslin}, and \cite{MR0294139},
that, in contrast to the existence of $\aleph_1$-Aronszajn trees, the existence of $\aleph_1$-Souslin trees is independent
of the usual axioms of set theory (\zfc).

As these objects proved incredibly useful and important, a systematic study of their consistency and interrelation was carried out.
Following standard conventions, we let $\tp_\kappa$ stand for the nonexistence of $\kappa$-Aronszajn trees (the \emph{tree property} at $\kappa$),
$\sh_\kappa$ stand for the nonexistence of $\kappa$-Souslin trees (the \emph{Souslin Hypothesis} at $\kappa$),
and $\ch_\lambda$ stand for $2^\lambda=\lambda^+$.
Two early results read as follows:
\begin{thm}[Specker, \cite{specker}] For every cardinal $\lambda$, $\ch_\lambda$ implies the failure of $\tp_{\lambda^{++}}$.\footnote{By a \emph{cardinal}, we always mean an infinite cardinal.}
\end{thm}
\begin{thm}[Jensen, \cite{jensen_fine_structure}] In G\"odel's constructible universe, $L$, for every regular, uncountable cardinal $\kappa$, the following are equivalent:
\begin{itemize}
\item $\tp_\kappa$;
\item $\sh_\kappa$;
\item $\kappa$ is a weakly compact cardinal.
\end{itemize}
\end{thm}

We remind the reader that a cardinal $\kappa$ is \emph{weakly compact} iff it is uncountable
and Ramsey's theorem holds at the level of $\kappa$, i.e., every graph of size $\kappa$ contains a clique or an anticlique of size $\kappa$.

Ever since Jensen's result, the general belief has been that the consistency of $\sh_\kappa$ for $\kappa$ of the form $\lambda^{++}$ requires the consistency of a weakly compact cardinal.
This conjecture is supported by the following later results:
\begin{thm}[Mitchell and Silver, \cite{MR0313057}]\label{thm13} The existence of a regular cardinal $\lambda$ for which
    $\tp_{\lambda^{++}}$ holds is equiconsistent with the existence of a weakly compact cardinal.
In particular, the consistency of a weakly compact cardinal gives the consistency of $\neg\ch_\lambda$ together with $\sh_{\lambda^{++}}$.
\end{thm}
\begin{thm}[Laver and Shelah, \cite{laver_shelah}]\label{thm14} For every cardinal $\lambda$, if there is a weakly compact
    cardinal above $\lambda$, then there is a forcing extension by a $\lambda^+$-directed closed forcing notion
    in which $\ch_\lambda$ and $\sh_{\lambda^{++}}$ both hold.
\end{thm}
\begin{thm}[Rinot, \cite{paper24}]\label{thm15} For every cardinal $\lambda$, if $\ch_\lambda$, $\ch_{\lambda^+}$, and
    $\sh_{\lambda^{++}}$ all hold, then $\lambda^{++}$ is a weakly compact cardinal in $L$.
\end{thm}
Whether the hypotheses of Theorem~\ref{thm15} are consistent, relative to any large cardinal assumption, is a major open problem.

In this paper, we are interested in a possible converse for Theorem~\ref{thm14}.
As of now, the best result in this direction gives a lower bound
of an inaccessible cardinal.\footnote{Recall that any weakly compact cardinal admits stationarily many Mahlo cardinals below it,
and any Mahlo cardinal admits stationarily many inaccessible cardinals below it.}

\begin{thm}[Shelah and Stanley, \cite{shelah_stanley}]\label{thm16} For every cardinal $\lambda$, if $\ch_\lambda$ and $\sh_{\lambda^{++}}$ both hold, then $\lambda^{++}$ is an inaccessible cardinal in $L$.
\end{thm}

Here, we establish the following.
\begin{thma} For every uncountable cardinal $\lambda$, if $\ch_\lambda$ and $\sh_{\lambda^{++}}$ both hold, then $\lambda^{++}$ is a Mahlo cardinal in $L$.
\end{thma}

The following table provides a clear summary of all of these results.

\begin{center}
\begin{tabular}{c|c|c|c|c|c|c}
Theorem&$\lambda$&$\ch_\lambda$&$\ch_{\lambda^+}$&$\sh_{\lambda^{++}}$&lower bound&upper bound\\
\hline
\ref{thm13}&regular&\xmark&\vmark&\vmark&&weakly compact\\
\hline
\ref{thm15}&arbitrary&\vmark&\vmark&\vmark&weakly compact&\\
\hline
\ref{thm14}&arbitrary&\vmark&\xmark&\vmark&&weakly compact\\
\ref{thm16}&arbitrary&\vmark&\xmark&\vmark&inaccessible&\\
A&uncountable&\vmark&\xmark&\vmark&Mahlo&\\
\end{tabular}
\end{center}

\subsection{Combinatorial constructions}
In order to prove Theorem A, we develop a general framework for carrying out combinatorial constructions.
It turns out that, in this and other applications, it is often desirable to be able to construct an object of size $\kappa^+$,
where $\kappa$ is a regular, uncountable cardinal, using approximations to that object of size $<\kappa$.
When one attempts to carry out such a construction just using the axioms of \zfc, though, one naturally
runs into problems: the construction seems to require $\kappa^+$ steps, but the approximations may become
too large after only $\kappa$ steps.

The usual way to attempt to overcome this problem is to assume, in addition to \zfc, certain nice combinatorial
features of $\kappa$ or $\kappa^+$. One such feature, whose definition is motivated by precisely such constructions, is
the existence of a $(\kappa, 1)$-morass (see \cite[\S4]{MR0376351} or \cite[Chapter VIII]{devlin_book}).
Velleman \cite{velleman_forcing}, and Shelah and Stanley \cite{shelah_stanley},
present frameworks for carrying out constructions of objects of size $\kappa^+$ using a
$(\kappa, 1)$-morass. In both instances, these frameworks take the form of forcing axioms which turn
out to be equivalent to the existence of morasses.

Another combinatorial assumption that can be helpful in these constructions is the existence of a diamond sequence. In a series of
papers on models with second order properties, culminating in a general treatment in \cite{shelah_laflamme_hart},
Shelah et al.\ develop a technique for using $\diamondsuit(\kappa)$ to build objects of size $\kappa^+$ out of
approximations of size $<\kappa$.
Ideas from these papers were used by Foreman, Magidor, and Shelah \cite{fms_ii} to prove, assuming
the consistency of a huge cardinal, the consistency of the existence of an ultrafilter $\mathcal U$ on $\omega_1$ such that
$|\omega^{\omega_1} / \mathcal U| = \aleph_1$, and later by Foreman \cite{foreman_dense_ideal} to prove, again assuming the
consistency of a huge cardinal, the consistency of the existence of an $\aleph_1$-dense ideal on $\aleph_2$.

In this paper, we present a framework for constructions of objects of size $\kappa^+$ using $\diamondsuit(\kappa)$
and $\square^B_\kappa$, a weakening of $\square_\kappa$ that, unlike $\square_\kappa$ itself, is implied by the
existence of a $(\kappa, 1)$-morass. As in \cite{velleman_forcing} and \cite{shelah_stanley}, our framework
takes the form of a forcing axiom. Specifically, in Section~\ref{forcing_sect},
we isolate a class of forcing notions $\mathcal P_\kappa$,
introduce the notion of a \emph{sharply dense system}, and formulate a forcing axiom, $\fa{\kappa}$,
that asserts that for every $\mathbb P$ from the class $\mathcal P_\kappa$ and every sequence
$\langle \mathcal{D}_i\mid i<\kappa\rangle$ of sharply dense systems, there is a filter $G$ on $\mathbb P$
that meets each $\mathcal{D}_i$ everywhere.

The last two sections of the paper are devoted to the proof of the following:

\begin{thmb} For every regular uncountable cardinal $\kappa$,
if $\diamondsuit(\kappa)$ and $\square^B_\kappa$ both hold, then so does $\fa{\kappa}$.
\end{thmb}

In Section~\ref{applications_sect}, we give a few simple applications of the forcing axiom $\fa{\kappa}$.
We open by pointing out that the Cohen forcing $\add(\kappa,\kappa^+)$
is a member of the class $\mathcal P_\kappa$.
Then, we show that $\fa\kappa$ entails $\kappa^{<\kappa}=\kappa$ and $\square^B_\kappa$.
This has three consequences. First, it shows that our square hypothesis in Theorem~B is optimal:

\begin{thmbp} Suppose that $\kappa$ is a regular, uncountable cardinal
and $\diamondsuit(\kappa)$ holds. Then the following are equivalent:
\begin{itemize}
\item $\square^B_\kappa$ holds;
\item $\fa{\kappa}$ holds.
\end{itemize}
\end{thmbp}

Second, by Shelah's theorem \cite{shelah_diamonds} stating that $\ch_\lambda$ entails $\diamondsuit(\lambda^+)$
for every uncountable cardinal $\lambda$, it gives cases in which the diamond hypothesis is optimal, as well:

\begin{thmbpp} For every successor cardinal $\kappa>\aleph_1$,
the following are equivalent:
\begin{itemize}
\item $\diamondsuit(\kappa)$ and $\square^B_{\kappa}$ both hold;
\item $\fa{\kappa}$ holds.
\end{itemize}
\end{thmbpp}

Third, it implies that $\fa{\kappa}$ entails the existence of a strong stationary coding set,
i.e., a stationary subset of $[\kappa^+]^{<\kappa}$
on which the map $x\mapsto\sup(x)$ is injective. This is of interest because the
existence of such a set was previously obtained by Shelah and Stanley
\cite{ShSt:167} from their forcing axiom $S_{\kappa}(\diamondsuit)$, which
is equivalent to the existence of a $(\kappa, 1)$-morass with a
built-in diamond sequence, and later (though earlier in terms of publication date) by Velleman
\cite{velleman_souslin} from the existence of a stationary simplified $(\kappa, 1)$-morass.

Section~\ref{tree_sect} is dedicated to the the study of super-Souslin trees.
For a cardinal $\lambda$, a \emph{$\lambda^{++}$-super-Souslin tree} is a $\lambda^{++}$-tree $(T,<_T)$ with a certain highly absolute
combinatorial property that ensures that $(T,<_T)$ has a $\lambda^{++}$-Souslin subtree in any
$\zfc$ extension $W$ of the universe $V$ that satisfies $\mathcal P^W(\lambda)=\mathcal P^V(\lambda)$
and $(\lambda^{++})^W = (\lambda^{++})^V$. These trees were introduced in a paper by Shelah and Stanley \cite{shelah_stanley}, where the existence of super-Souslin trees
provided the primary application of the forcing axiom isolated in that paper.
In particular, they proved that the existence of a $\lambda^{++}$-super-Souslin tree follows from the existence of a $(\lambda^+, 1)$-morass
together with $\ch_\lambda$. In \cite{velleman_forcing} and \cite{ShSt:167} the same hypotheses are shown to entail the existence of a $\lambda^{++}$-super-Souslin tree which is moreover $\lambda^+$-complete. Here, we prove the following analogous result.

\begin{thmc}   For every cardinal $\lambda$, $\fa{\lambda^+}$ entails
the existence of a $\lambda^+$-complete $\lambda^{++}$-super-Souslin tree.
\end{thmc}

By Theorems B and C, and the fact that for any super-Souslin tree $(T,<_T)$, there exists some $x\in T$ such that $(x^\uparrow,<_T)$ is Souslin, we obtain:

\begin{cor} For every cardinal $\lambda$, if $\diamondsuit(\lambda^+)$ and $\square_{\lambda^+}^B$ both hold,
then there is a $\lambda^+$-complete $\lambda^{++}$-Souslin tree.
\end{cor}

Recalling Jensen's theorem \cite{jensen_fine_structure} stating that if $\square_\kappa$ fails,
then $\kappa^{+}$ is a Mahlo cardinal in $L$,
and Shelah's theorem \cite{shelah_diamonds} stating that $\ch_\lambda$ entails $\diamondsuit(\lambda^+)$
for every uncountable cardinal $\lambda$,
we see that Theorem~A follows from Corollary~1.

We also obtain a corollary concerning partition relations.
Recall that, for ordinals $\alpha, \beta$, and $\gamma$, the statement $\alpha\rightarrow(\beta,\gamma)^2$
asserts that, for every coloring $c:[\alpha]^2\rightarrow\{0,1\}$,
either there exists $B\s\alpha$ of order type $\beta$ which is $0$-monochromatic,
or there exists $C\s\alpha$ of order type $\gamma$ which is $1$-monochromatic.
By a recent theorem of Raghavan and Todorcevic \cite{1602.07901},
the existence of a $\kappa^+$-Souslin tree entails $\kappa^+\nrightarrow(\kappa^+,\log_\kappa(\kappa^+)+2)^2$,
where $\log_\kappa(\kappa^+)$ denotes the least cardinal $\nu$ such that $\kappa^\nu > \kappa$. We thus obtain
the following corollary:

\begin{cor} Suppose that $\lambda$ is an uncountable cardinal. If $\ch_\lambda$ and
  $\lambda^{++}\rightarrow(\lambda^{++},\lambda^++2)^2$ both hold, then $\lambda^{++}$
  is a Mahlo cardinal in $L$.
\end{cor}

Note that by a theorem of Erd\H{o}s and Rado, $\ch_\lambda$ entails $\lambda^{++}\rightarrow(\lambda^{++},\lambda^++1)^2$.

\subsection{Notations and conventions}
We write \coi as a shorthand for ``continuous, order-preserving injection.''
In particular, a \coi is a map $\pi$ from a set of ordinals into the ordinals
such that $\pi$ is continuous, order-preserving, and injective, and, moreover,
$\dom(\pi)$ is closed in its supremum. Thus, when we write, for example,
``$\pi:y \rightarrow \kappa^+$ is a \coi,'' it is implicit that $y$ is closed
in its supremum.
For ordinals $\theta<\mu$, let ${\mu \choose \theta}:=\{\im(\pi)\mid \pi:\theta\rightarrow\mu~\allowbreak\text{is a \coi}\}$,
i.e., ${\mu \choose \theta}$ consists of all closed copies of $\theta$ in $\mu$.

For a set of ordinals $x$, $\otp(x)$ denotes the order type of $x$ and, for all $i < \otp(x)$, $x(i)$ denotes the unique element $\alpha$ of $x$ such that $\otp(x \cap \alpha) = i$.
We write $\ssup(x) := \sup\{\alpha + 1 \mid \alpha \in x\}$,
$\acc(x) := \{\alpha \in x \mid \sup(x \cap \alpha) = \alpha>0\}$,
$\nacc(x):=x\setminus\acc(x)$,
$\acc^+(x):=\{\alpha<\ssup(x)\mid \sup(x\cap\alpha)=\alpha>0\}$,
and $\cl(x):=x\cup\acc^+(x)$.
By convention, $\ssup(\emptyset)=\sup(\emptyset)=0$.
For sets of ordinals $x$ and $y$, we write $x \sq y$ iff $y$ is an end-extension of $x$,
i.e., $y \cap \ssup(x) = x$. For cardinals  $\lambda < \mu$,
let $E^\mu_\lambda := \{\alpha < \mu \mid \cf(\alpha) = \lambda\}$,
let $E^\mu_{<\lambda} := \{\alpha < \mu \mid \cf(\alpha) < \lambda\}$,
let $[\mu]^{<\lambda}:=\{x\s\mu\mid |x|<\lambda\}$,
and let $[\mu]^2:=\{ (\alpha,\beta)\mid \alpha<\beta<\mu\}$.
Also, let $H_\mu$ denote the collection of all sets of hereditary cardinality less than $\mu$.

Throughout the paper, $\kappa$ stands for an arbitrary regular, uncountable cardinal.
For simplicity, the reader may assume that $\kappa=\aleph_1$.

\section{The forcing axiom} \label{forcing_sect}
We begin by introducing the class $\mathcal{P}_\kappa$ of forcing notions that will be of interest.

\begin{defn} \label{class_def}
  $\mathcal{P}_\kappa$ consists of all triples $(\bb{P}, \leq_{\bb{P}}, \bb{Q})$ such that
  $(\bb{P}, \leq_{\bb{P}})$ is a forcing notion, $\one\in \bb{Q} \subseteq \bb{P}$,
  and all of the following requirements hold.
  \begin{enumerate}
    \item\label{c1}\textbf{(Realms)} For all $p \in \bb{P}$, there is a unique $x_p \in [\kappa^+]^{<\kappa}$, which we
      refer to as the \emph{realm} of $p$. The map $p\mapsto x_p$
      is a projection from $(\bb{P},\le_{\bb{P}})$ to $([\kappa^+]^{<\kappa},\supseteq)$:
      \begin{enumerate}
    \item $x_{\one}=\emptyset$;
    \item for all $q \leq_\bb{P} p$, we have $x_q \supseteq x_p$;
    \item\label{c1.5} for all $p\in\bb{P}$ and $x \in [\kappa^+]^{<\kappa}$ with $x\supseteq x_p$, there is $q \leq_{\bb{P}} p$ with $x_q = x$.
      \end{enumerate}
    \item\label{c2}\label{c3}\textbf{(Scope)} For all $y \subseteq \kappa^+$, let $\bb{P}_y := \{p \in \bb{P} \mid x_p \subseteq y\}$
      and $\bb{Q}_y := \bb{Q} \cap \bb{P}_y$. Then $\bb{P}_\emptyset=\{\one\}$ and $\bb{P}_\kappa \subseteq H_\kappa$.
    \item\label{c4}\textbf{(Actions of \cois)} For every $y \subseteq \kappa^+$ and every \coi $\pi:y \rightarrow \kappa^+$,
      $\pi$ acts on $\bb{P}_y$ in such a way that, for all $p,q \in \bb{P}_y$:
      \begin{enumerate}
      \item\label{c4a}\label{c4c} $\pi.p$ is in $\bb{P}$ with $x_{\pi.p} = \pi``x_p$,
      and if $p\in\bb{Q}_y$, then $\pi.p$ is in $\bb{Q}$;
      \item\label{c4d} $\pi.q \leq_{\bb{P}} \pi.p$ iff $q \leq_{\bb{P}} p$;
      \item\label{c4g} if $\pi$ is the identity map, then $\pi.p=p$;
      \item\label{c4f} if $\pi':y' \rightarrow \kappa^+$ is a \coi with $\im(\pi)\s y'$, then  $\pi'.(\pi.p)=(\pi'\circ\pi).p$;
      \item if $\pi':y' \rightarrow \kappa^+$ is a \coi with $x_p\s y'$, then $\pi \restriction x_p = \pi' \restriction x_p$ implies that $\pi.p = \pi'.p$.
      \end{enumerate}
    \item\label{c5}\textbf{(Restrictions)} For all $p \in \bb{P}$ and $\alpha < \kappa^+$, there is a unique
      $\leq_{\bb{P}}$-least condition $r$ such that $x_r = x_p \cap \alpha$
      and $p \leq_{\bb{P}} r$. This condition $r$ is referred to as $p \rest \alpha$.
      Moreover:
      \begin{enumerate}
        \item if $p \in \bb{Q}$, then $p \rest \alpha \in \bb{Q}$;
        \item\label{c8} if $q \leq_{\bb{P}} p$, then $q \rest \alpha \leq_{\bb{P}} p \rest \alpha$.
      \end{enumerate}
    \item\label{c12}\textbf{(Vertical limits)} Suppose that $\xi < \kappa$ and $\langle p_\eta \mid \eta < \xi \rangle$ is a sequence of conditions from
      $\bb{P}$ such that, for all $\eta < \eta' < \xi$, we have $p_\eta = p_{\eta'} \rest \ssup(x_{p_\eta})$.
      Then there is a unique condition $p \in \bb{P}$ such that $x_p = \bigcup_{\eta < \xi} x_{p_\eta}$
      and, for all $\eta < \xi$, $p_\eta = p \rest \ssup(x_{p_\eta})$. Moreover, if $p_\eta \in \bb{Q}$
      for all $\eta < \xi$, then $p \in \bb{Q}$.
    \item\label{c9}\textbf{(Sharpness)} For all $q \in \bb{Q}$, $x_q$ is closed in its supremum. Moreover, for all $p \in \bb{P}$,
      there is $q \leq_{\bb{P}} p$ with $x_q = \cl(x_p)$ such that $q \in \bb{Q}$.

    \item\label{c10}\textbf{(Controlled closure)} Suppose that $\xi < \kappa$ and
      $\langle q_\eta \mid \eta < \xi \rangle$ is a decreasing sequence of conditions from $\bb{Q}$.
    Let $x := \bigcup_{\eta < \xi} x_{q_\eta}$.
    Suppose that $\alpha<\ssup(x)$ and that $r \in \bb{Q}_{\ssup(x \cap \alpha)}$ is a lower bound for
    $\langle q_\eta \rest \alpha \mid \eta < \xi \rangle$. Then there is $q \in \bb{Q}$ such that:
      \begin{enumerate}
        \item $q \rest \ssup(x \cap \alpha) = r$;
        \item $x_q = \cl(x_r \cup x)$;
        \item $q$ is a lower bound for $\langle q_\eta \mid \eta < \xi \rangle$.
      \end{enumerate}
    \item\label{c11}\textbf{(Amalgamation)} For all $p \in \bb{Q}$, $\alpha < \ssup(x_p)$, and $q \in \bb{P}_\alpha$ with $q \leq_{\bb{P}} p \rest \alpha$,
      we have that $p$ and $q$ have a unique $\leq_{\bb{P}}$-greatest lower bound
      $r$. Moreover, it is the case that $x_r = x_q \cup x_p$ and $r \rest \alpha = q$.
  \end{enumerate}
\end{defn}

We now introduce the class of families of dense sets that we will be interested in meeting.

\begin{defn}[Sharply dense set]
  Suppose that $(\bb{P}, \leq_{\bb{P}}, \bb{Q}) \in \mathcal{P}_\kappa$ and $D$ is a nonempty subset of $\bb{P}$.
   Denote $x_D:=\bigcap\{x_p\mid p\in D\}$.
We say that $D$ is  \emph{sharply dense} iff
  for every $p \in \mathbb{P}$, there is $q \in D$ with $q \leq_{\bb{P}} p$ such that $x_q = \cl(x_p \cup x_D)$.
\end{defn}

\begin{defn}[Sharply dense system]
  Suppose that $(\bb{P}, \leq_{\bb{P}}, \bb{Q}) \in \mathcal{P}_\kappa$.
  We say that  $\mathcal D\s\mathcal P(\mathbb P)$ is a \emph{sharply dense system}
  iff there exists an ordinal $\theta_{\mathcal D}<\kappa$ such that $\mathcal D$ is of the form $\{D_x \mid x \in{\kappa^+\choose\theta_{\mathcal D}}\}$,
  where for all $x \in{\kappa^+\choose\theta_{\mathcal D}}$:
    \begin{itemize}
    \item $D_x$ is sharply dense with $x_{D_x}=x$;
    \item  for every $p \in \mathbb{P}$,       and every \coi $\pi:y \rightarrow \kappa^+$ with $x \subseteq x_p \subseteq y$,
      we have $p \in D_x$ iff $\pi.p \in D_{\pi``x}$.
  \end{itemize}
\end{defn}

\begin{defn}
  Suppose that $(\bb{P}, \leq_{\bb{P}}, \bb{Q}) \in \mathcal{P}_\kappa$ and
  $\mathcal{D}$ is a
  sharply dense system. We say that a filter $G$ on $\bb{P}$
  \emph{meets $\mathcal{D}$ everywhere} iff, for all $D\in\mathcal D$,
  $G \cap D \neq \emptyset$.
\end{defn}

We are now ready to formulate our forcing axiom for sharply dense systems.

\begin{defn}
  $\fa{\kappa}$ is the assertion that, for every $(\bb{P}, \leq_\bb{P}, \bb{Q}) \in \mathcal{P}_\kappa$ and every collection
  $\{\mathcal{D}_i \mid i < \kappa\}$ of sharply dense systems, there exists a
  filter $G$ on $\mathbb{P}$ such that, for all $i<\kappa$, $G$ meets $\mathcal D_i$ everywhere.
\end{defn}

\section{Applications}\label{applications_sect}
In this section we present a few applications of $\fa\kappa$.
Just before that, let us point out two features of members of the class $\mathcal P_\kappa$.

\begin{prop}\label{prop31} Suppose that $(\bb{P}, \leq_{\bb{P}}, \bb{Q}) \in \mathcal{P}_\kappa$. Then:
\begin{enumerate}
\item $( \mathbb Q, \leq_{\bb{P}})$ is $\kappa$-closed.
\item For all $x\s\kappa^+$, denote $D_x:=\{ q\in\mathbb Q\mid x_q\supseteq x\}$.
Then, for all $\theta < \kappa$, $\{ D_x\mid x\in{\kappa^+\choose\theta}\}$ is a sharply dense system.
\end{enumerate}
\end{prop}
\begin{proof}
(1) Suppose that $\xi < \kappa$ and $\vec q=\langle q_\eta \mid \eta < \xi \rangle$ is a decreasing sequence of conditions from $\bb{Q}$.
Note that if $x := \bigcup_{\eta < \xi} x_{q_\eta}$ is empty, then $\one$ is a lower bound for $\vec q$, so we may assume that $x$ is nonempty.
Since $\one\in\bb{Q}$ and $\bb{P}_0=\{\one\}$, we
infer from Clause~\eqref{c5} of Definition~\ref{class_def} that $\{q_\eta\rest 0\mid \eta<\xi\}=\bb{Q}_0=\{\one\}$.
So, by Clause~\eqref{c10} of Definition~\ref{class_def}, using $\alpha:=0$ and $r:=\one$, we infer that $\vec q$ admits a lower bound.

(2) By Clauses \eqref{c4c} and \eqref{c9} of Definition~\ref{class_def}.
\end{proof}

Next, we show that the actions of \cois behave as expected with respect to the restriction operation.

\begin{prop}
  Suppose that $(\bb{P}, \leq_{\bb{P}}, \bb{Q}) \in \mathcal{P}_\kappa$, $p \in \bb{P}$, $\alpha \in x_p$, and
  $\pi : y \rightarrow \kappa^+$ is a \coi with $x_p \s y \s \kappa^+$.
  Then $\pi.(p \rest \alpha) = (\pi.p) \rest \pi(\alpha)$.
\end{prop}

\begin{proof}
  Let $r := \pi.(p \rest \alpha)$.
  Since $p \leq_\bb{P} p \rest \alpha$, Clause~\eqref{c4d} (of Definition~\ref{class_def}) implies that $\pi.p \leq_\bb{P} r$.
  In addition, by Clauses \eqref{c4a} and \eqref{c5}, and since $\alpha\in y$, we have:
$$x_r = \pi``x_{p \rest \alpha} = \pi``(x_p \cap \alpha)= \pi``x_p \cap \pi``\alpha= \pi``x_p \cap\pi(\alpha)= x_{\pi.p} \cap \pi(\alpha).$$
  This shows that $r$ is a candidate for being $(\pi.p)\rest\pi(\alpha)$.
  To finish the proof, fix an arbitrary $q \in \bb{P}$ such that
  $x_q = x_{\pi.p}\cap\pi(\alpha)$ and $\pi.p \leq_\bb{P} q$.
  We have to verify that $r \leq_\bb{P} q$.

  Let $\pi':=\{(\delta,\varepsilon)\mid (\varepsilon,\delta)\in \pi\}$, so that $\pi'$ is a \coi and $\pi'\circ\pi$ and $\pi\circ\pi'$ are the identity maps
  on their respective domains.
  Since $\pi.p \leq_\bb{P} q$ and $x_q\s\im(\pi)$, and by Clauses \eqref{c4g} and \eqref{c4f},
  we have $p=\pi'.(\pi.p) \leq_\bb{P} \pi'.q$.
  Moreover, $x_{\pi'.q} = x_p \cap \alpha$, so, by Clause~\eqref{c5},
  $p \rest \alpha \leq_\bb{P} \pi'.q$. Now another application of Clauses \eqref{c4d} and \eqref{c4g} yields $r \leq_{\bb{P}} q$.
\end{proof}

\subsection{A warm-up example}
Let us point out that $\mathbb P:=\add(\kappa,\kappa^+)$ belongs to the class $\mathcal P_\kappa$.
Specifically, $p\in\mathbb P$ iff $p$ is a function from a subset of $\kappa^+\times\kappa$ to $2$
and $|p|<\kappa$. Let $p \le_{\bb{P}} q$ iff $p\supseteq q$. Let $x_p:=\{\beta\in\kappa^+\mid \exists\eta[(\beta,\eta)\in\dom(p)]\}$.
Let $\mathbb Q:=\{ p\in\mathbb P\mid x_p=\cl(x_p)\}$.
Whenever $\pi$ is a \coi from a subset of $\kappa^+$ to $\kappa^+$
and $p\in\mathbb P_{\dom(\pi)}$, we let $\pi.p:=\{ ((\pi(\beta),\eta),i)\mid ((\beta,\eta),i)\in p\}$.
We also let $p\rest\alpha:=\{((\beta,\eta),i)\in p\mid \beta<\alpha\}$.
The reader is now encouraged to verify that, with this definition, $(\mathbb P,\le_{\bb {P}},\mathbb Q) \in \mathcal P_\kappa$.

\subsection{Cardinal arithmetic}
In this subsection, we identify a simple member of $\mathcal{P}_\kappa$ and use it to prove that $\fa{\kappa}$
implies $\kappa^{<\kappa} = \kappa$.

\begin{defn}
  $\mathbb{P}$ consists of all pairs $p = (x, f)$ such that:
  \begin{enumerate}
    \item $x \in [\kappa^+]^{<\kappa}$;
    \item $f$ is a function such that:
      \begin{enumerate}
        \item $|f| < \kappa$;
        \item $\dom(f) \subseteq x \times \kappa$;
        \item for all $(\beta, \eta) \in \dom(f)$, we have $f(\beta, \eta) \subseteq \beta\cap x$.
      \end{enumerate}
  \end{enumerate}
\end{defn}

The coordinates of a condition $p \in \bb{P}$ will often be identified as $x_p$ and $f_p$, respectively.

\begin{defn}
  For all $p, q \in \bb{P}$, we let $q \leq_{\bb{P}} p$ iff $x_q \supseteq x_p$ and $f_q \supseteq f_p$.
\end{defn}

\begin{defn}
  $\bb{Q}:=\{ p \in \bb{P} \mid x_p=\cl(x_p)$\}.
\end{defn}

\begin{defn}
  Suppose that $\pi$ is a \coi from a subset of $\kappa^+$ to $\kappa^+$. For each $p \in \bb{P}_{\dom(\pi)}$, we
  let $\pi.p$ be the condition $(x,f)$ such that:
  \begin{enumerate}
    \item $x = \pi``x_p$;
    \item $f = \{((\pi(\beta), \eta), \pi``z) \mid ((\beta, \eta), z) \in f_p\}$.
  \end{enumerate}
\end{defn}

\begin{defn}
  Suppose that $p \in \bb{P}$ and $\alpha < \kappa^+$. Then we define $p \rest \alpha$ to be the condition
  $(x,f)$ such that:
  \begin{enumerate}
    \item $x = x_p \cap \alpha$;
    \item $f = \{((\beta, \eta), z) \in f_p \mid \beta < \alpha\}$.
  \end{enumerate}
\end{defn}

It is readily verified that, with these definitions, $(\bb{P}, \leq_{\bb{P}}, \bb{Q})$ is a member of $\mathcal{P}_\kappa$.

\begin{thm}
  Suppose $\kappa^{<\kappa}>\kappa$. Then $(\bb{P}, \leq_{\bb{P}}, \bb{Q})$ witnesses that $\fa{\kappa}$ fails.
\end{thm}

\begin{proof}
  We commence with a simple observation.
  \begin{subclaim} There exists a cardinal $\lambda<\kappa$ for which  $|{\lambda\choose\lambda}|>\kappa$.
  \end{subclaim}
  \begin{proof} Since $\kappa$ is regular, we have $\kappa^{<\kappa}=\sum_{\lambda<\kappa}\lambda^\lambda$.
  So, since $\kappa^{<\kappa}\ge\kappa^+$ and $\kappa^+$ is regular, we may fix a cardinal
  $\lambda<\kappa$ such that $\lambda^\lambda\ge\kappa^+$.
  For every $A\s\lambda$, let $$C_A:=\acc(\lambda)\cup\{\alpha+1\mid \alpha\in A\}.$$
  Then $A\mapsto C_A$ is an injection from $\mathcal P(\lambda)$ to ${\lambda\choose\lambda}$, and we are done.
  \end{proof}

  Fix a cardinal $\lambda<\kappa$ such that $|{\lambda\choose\lambda}|>\kappa$. For each $x \in {\kappa^+ \choose \lambda+1}$, let $D_x$ be the set of all conditions $(x_p,f_p) \in \bb{P}$ such that:
  \begin{itemize}
    \item $x \subseteq x_p$;
    \item there is $\eta < \kappa$ with $(\max(x), \eta) \in \dom(f_p)$ such that  $$f_p(\max(x), \eta) = x \cap \max(x).$$
  \end{itemize}
  Evidently, $\mathcal{D} := \{D_x \mid x \in {\kappa^+ \choose \lambda+1}\}$
  is a sharply dense system.

  Towards a contradiction, suppose that $\fa{\kappa}$ holds. In particular, there exists a filter
  $G$ on $\bb{P}$ that meets $\mathcal{D}$ everywhere.
  Let $f := \bigcup_{p \in G} f_p$, so that $f$ is a function from a (possibly proper) subset of $\kappa^+\times\kappa$ to $\mathcal P(\kappa^+)$.
  Put $\Lambda:=\{ f(\lambda,\eta)\mid \exists\eta<\kappa[(\lambda,\eta)\in \dom(f)]\}$.
  Clearly, $|\Lambda|\le\kappa$.
  Finally, let $C\in{\lambda\choose\lambda}$ be arbitrary.
  Since $C\cup \{\lambda\} \in {\kappa^+ \choose \lambda+1}$,
  we have $G \cap D_{C\cup \{\lambda\}} \neq \emptyset$, and hence $C\in \Lambda$.
  It follows that ${\lambda\choose\lambda}\s\Lambda$, contradicting the fact that $|{\lambda\choose\lambda}|>\kappa\ge|\Lambda|$.
\end{proof}

\begin{coro}\label{cardinalarithmetic} $\fa\kappa$ entails $\kappa^{<\kappa}=\kappa$.\qed
\end{coro}
\subsection{Baumgartner's square} \label{recovering_sect}
In unpublished work, Baumgartner introduced the principle $\square^B_\kappa$,
which is a natural weakening of Jensen's $\square_\kappa$ principle.

\begin{defn}
  A $\square^B_\kappa$-sequence is a sequence $\langle C_\beta \mid \beta \in \Gamma \rangle$ such that:
  \begin{enumerate}
    \item $E^{\kappa^+}_\kappa \subseteq \Gamma \subseteq \acc(\kappa^+)$;
    \item for all $\beta \in \Gamma$, $C_\beta$ is club in $\beta$ and $\otp(C_\beta) \leq \kappa$;
    \item for all $\beta\in \Gamma$ and all $\alpha\in \acc(C_\beta)$, we have $\alpha\in\Gamma$ and
      $C_{\alpha} = C_\beta\cap\alpha$.
  \end{enumerate}
  The principle $\square^B_\kappa$ asserts the existence of a $\square^B_\kappa$-sequence.\footnote{Note that
  $\square_\kappa^B$ is equivalent to the principle $\square_\kappa(\kappa^+,\sq_\kappa)$ from \cite[\S1]{paper29}.}
\end{defn}

Some basic facts
about $\square^B_\kappa$ can be found in \cite{velleman_forcing}, where it goes by the name
``weak $\square_\kappa$.'' In particular, it is shown in \cite{velleman_forcing} that
$\square^B_\kappa$ follows from the existence of a $(\kappa, 1)$-morass.

\begin{thm} \label{recovering_thm}
  Suppose that $\fa{\kappa}$ holds. Then so does $\square^B_\kappa$.
\end{thm}

The rest of this subsection is devoted to proving Theorem~\ref{recovering_thm}. We must first identify a relevant member of
$\mathcal{P}_\kappa$, which will be a slight modification of the poset used to add
$\square^B_\kappa$ in \cite[\S1.3]{velleman_forcing}.

\begin{defn} \label{square_p_def}
  $\bb{P}$ consists of all pairs $p = (x, f)$ such that:
  \begin{enumerate}
    \item $x \in [\kappa^+]^{<\kappa}$;
    \item $f$ is a function from $x$ to $\mathcal P(x)$ such that for all $\beta\in x$:
      \begin{enumerate}
        \item   $f(\beta)$ is a closed subset of $\beta$;\footnote{We say that $c$ is a \emph{closed subset of $\beta$}
iff $c\s\beta$ and for every $\alpha<\beta$, $c\cap\alpha\neq\emptyset\implies\sup(c\cap\alpha)\in c$.}
        \item for all $\alpha \in \acc(f(\beta))$, we have $f(\alpha) = f(\beta) \cap \alpha$.
      \end{enumerate}
  \end{enumerate}
\end{defn}
  The coordinates of a condition $p \in \bb{P}$ will often be identified as $x_p$ and $f_p$, respectively.
\begin{defn} \label{square_order_def}
  For all $p,q \in \bb{P}$, we let $q \leq_\bb{P} p$ iff:
  \begin{itemize}
    \item $x_p \subseteq x_q$;
    \item for all $\beta\in x_p$, we have $f_p(\beta) \sq f_q(\beta)$;
    \item for all $\beta\in x_p$, if $\sup(x_p \cap \beta) = \beta$, then $f_q(\beta) = f_p(\beta)$.
  \end{itemize}
\end{defn}

\begin{defn} \label{square_q_def}
  $\bb{Q}$ is the set of all conditions $p \in \bb{P}$ such that:
  \begin{enumerate}
    \item $x_p=\cl(x_p)$;
    \item for all $\beta \in \nacc(x_p) \setminus \{\min(x_p)\}$, we have $\max(f_p(\beta)) = \max(x_p \cap \beta)$.
  \end{enumerate}
\end{defn}

In order to show that $(\bb{P}, \leq_\bb{P}, \bb{Q}) \in \mathcal{P}_\kappa$, we must define the actions of
\cois on $\bb{P}$ and a restriction operation.

\begin{defn}
  Suppose that $\pi$ is a \coi from a subset of $\kappa^+$ to $\kappa^+$.
    For each $p \in \bb{P}_{\dom(\pi)}$, we define $\pi.p$ to be the condition $(x,f) \in \bb{P}$ such that:
  \begin{enumerate}
    \item $x = \pi``x_p$;
    \item for all $\alpha \in x_p$, we have $f(\pi(\alpha)) = \pi``f_p(\alpha)$.
  \end{enumerate}
\end{defn}

\begin{defn}
  Suppose that $p \in \bb{P}$ and $\alpha < \kappa^+$. Then $p \rest \alpha$ is the condition $(x,f) \in \bb{P}$ such that
  $x = x_p \cap \alpha$ and $f = f_p \restriction x$.
\end{defn}

Naturally, for each $p \in \bb{P}$, we let $x_p$ denote the realm of $p$.
With these definitions, it is immediate that $(\bb{P}, \leq_\bb{P}, \bb{Q})$ satisfies Clauses \eqref{c1}--\eqref{c12}
of Definition~\ref{class_def}. We now verify Clauses \eqref{c9}--\eqref{c11}, in order.

\begin{lemma}
  Suppose that $p \in \bb{P}$. Then there is $q \in \bb{Q}$ with $q \leq_\bb{P} p$ such that $x_q = \cl(x_p)$.
\end{lemma}
\begin{proof}
    Set $x_q := \cl(x_p)$, so that $\nacc(x_q)=\nacc(x_p)$ and $\acc(x_q)\supseteq\acc(x_p)$.
    Next, define $f_q:x_q\rightarrow\mathcal P(x_q)$ by stipulating:
    $$f_q(\alpha):=\begin{cases}
    f_p(\alpha)\cup\{\max(x_q\cap\alpha)\} &\text{if }\alpha\in\nacc(x_q)\setminus\{\min(x_q)\};\\
    \emptyset &\text{if }\alpha\in x_q\setminus x_p;\\
    f_p(\alpha) &\text{otherwise}.
    \end{cases}$$

    It is clear that $q := (x_q, f_q)$ is as desired.
\end{proof}

\begin{lemma}
  Suppose that $\xi < \kappa$ and $\langle q_\eta \mid \eta < \xi \rangle$ is a decreasing sequence
  of conditions from $\bb{Q}$. Let $x := \bigcup_{\eta < \xi} x_{q_\eta}$, and suppose that $\alpha
  < \ssup(x)$ and $r \in \bb{Q}_{\ssup(x \cap \alpha)}$ is a lower bound for $\langle q_\eta \rest \alpha \mid \eta < \xi \rangle$.
  Then there is $q \in \bb{Q}$ such that:
  \begin{enumerate}
    \item $q \rest \ssup(x \cap \alpha) = r$;
    \item $x_q = \cl(x_r \cup x)$;
    \item $q$ is a lower bound for $\langle q_\eta \mid \eta < \xi \rangle$.
  \end{enumerate}
\end{lemma}

\begin{proof}
  We will construct a condition $q = (x_q, f_q)$ as desired. We are required to let $x_q := \cl(x_r \cup x)$
  and to ensure that $f_q \restriction x_r := f_r$. As $x_r \sq x_q$, it remains to determine $f_q \restriction (x_q \setminus \ssup(x \cap \alpha))$.
  We will define $f_q(\beta)$ by recursion on $\beta \in (x_q \setminus \ssup(x \cap \alpha))$, maintaining the hypothesis
  that $(x_q \cap (\beta + 1), f_q \restriction (\beta + 1))$ is an element of $\bb{Q}$ and a lower bound for
  $\langle q_\eta \rest (\beta + 1) \mid \eta < \xi \rangle$. For notational ease, if
  $\beta \in \nacc(x_q) \setminus \{\min(x_q)\}$, then let $\beta^- := \max(x_q \cap \beta)$.

  $\br$ If $\beta \in \acc(x)$, then fix $\eta_\beta < \xi$
  such that $\beta \in x_{q_{\eta_\beta}}$, and let $f_q(\beta) := \bigcup_{\eta \in [\eta_\beta, \xi)} f_{q_\eta}(\beta)$.
  There are two possibilities to consider here. If there is $\eta^* \in [\eta_\beta, \xi)$
  such that $\sup(x_{q_{\eta^*}} \cap \beta) = \beta$, then it follows from Definition~\ref{square_order_def}
  that $f_q(\beta) = f_{q_{\eta^*}}(\beta)$.

  If, on the other hand, there is no such $\eta^*$, then the fact that
  each $x_{q_\eta}$ is closed in its supremum implies that, for all $\eta \in [\eta_\beta, \xi)$,
  we have $\beta \in \nacc(x_{q_\eta})$ and hence $\max(f_{q_\eta}(\beta)) = \max(x_{q_\eta} \cap \beta)$.
  Since $\beta\in\acc(x)$, it then follows that $f_q(\beta)$ is club in $\beta$.

  $\br$ If $\beta \in \acc^+(x) \setminus x$, then let
  $\gamma := \min(x \setminus (\beta + 1))$. There is $\eta_\beta < \xi$ such that,
  for all $\eta \in [\eta_\beta, \xi)$, we have $\gamma \in x_{q_\eta}$ and $x_{q_\eta} \cap \beta \neq \emptyset$.
  For all such $\eta$, let $\delta_\eta := \max(x_{q_\eta} \cap \beta)$. It follows that
  $\sup\{\delta_\eta \mid \eta \in [\eta_\beta, \xi)\} = \beta$ and, for all $\eta \in [\eta_\beta, \xi)$, we have
  $\max(f_{q_\eta}(\gamma)) = \delta_\eta$. We can therefore let $f_q(\beta) := \bigcup_{\eta \in [\eta_\beta, \xi)} f_{q_\eta}(\gamma)$.

  $\br$ If $\beta \in \nacc(x)$ and $\beta^- \notin x$, then, by the construction in
  the previous case, we have $f_q(\beta^-) = \bigcup_{\eta < \xi} f_{q_\eta}(\beta)$.
  We can therefore let $f_q(\beta) := f_q(\beta^-) \cup \{\beta^-\}$.

  $\br$ If $\beta \in \nacc(x)$ and $\beta^- \in x$, then there is
  $\eta_\beta < \xi$ such that $\{\beta, \beta^-\} \subseteq x_{q_{\eta_\beta}}$. But then,
  for all $\eta \in [\eta_\beta, \xi)$, we have $f_{q_\eta}(\beta) = f_{q_{\eta_\beta}}(\beta)$
  and $\max(f_{q_\eta}(\beta)) = \beta^-$. We can therefore let $f_q(\beta) := f_{q_{\eta_\beta}}(\beta)$.

  It is easily verified that $q$, constructed in this manner, is as desired.
\end{proof}

\begin{lemma}
  Suppose that $p \in \bb{Q}$, $\alpha < \ssup(x_p)$, $q \in \bb{P}_\alpha$, and $q \leq_\bb{P} p\rest\alpha$.
  Then $p$ and $q$ have a $\leq_\bb{P}$-greatest lower bound, $r$. Moreover, we have $x_r = x_p \cup x_q$
  and $r \rest \alpha = q$.
\end{lemma}

\begin{proof}
  Let $x_r := x_p \cup x_q$, so that $x_r\cap\alpha=x_q$. Define $f_r:x_r\rightarrow\mathcal P(x_r)$ by stipulating:
  $$f_r(\beta):=\begin{cases}
  f_q(\beta)&\text{if }\beta<\alpha;\\
  f_p(\beta)&\text{otherwise}.\end{cases}$$

  To see that $r := (x_r, f_r)$ is a condition, we fix arbitrary $\beta\in x_r$ and $\gamma\in\acc(f_r(\beta))$,
  and verify that $f_r(\gamma)=f_r(\beta)\cap\gamma$.
  To avoid trivialities, suppose that $\beta\ge\alpha>\gamma$.
  Since $f_r(\beta)=f_p(\beta)\s x_p$, we have $\sup(x_p\cap\gamma)=\gamma$, so, since $q \leq_\bb{P} p\rest\alpha$,
  we infer that $f_q(\gamma)=f_p(\gamma)=f_p(\beta)\cap\gamma$, i.e., $f_r(\gamma)=f_r(\beta)\cap\gamma$.

  It is now readily checked that $r$ has the desired properties.
\end{proof}

It follows that $(\bb{P}, \leq_\bb{P}, \bb{Q}) \in \mathcal{P}_\kappa$.
For each $x\in{\kappa^+\choose3}$, let $D_x := \{p \in \bb{Q} \mid x_p \supseteq x\}$.
By Proposition~\ref{prop31}(2), $\mathcal{D} :=
\{D_x \mid x \in {\kappa^+\choose 3}\}$ is a sharply dense system, so we can apply $\fa{\kappa}$
to obtain a filter $G$ on $\bb{P}$ that meets $\mathcal{D}$ everywhere. For all
$\beta \in E^{\kappa^+}_\kappa$, let $C_\beta := \bigcup\{f_p(\beta) \mid p \in G, ~ \beta \in x_p\}$.
Note that for all $p\in G$ and $\beta\in x_p$, we have $|f_p(\beta)|\le|x_p|<\kappa$.

\begin{claim}  Suppose that $\beta,\gamma \in E^{\kappa^+}_\kappa$. Then:
\begin{enumerate}
\item    $C_\beta$ is club in $\beta$ and $\otp(C_\beta) = \kappa$;
\item For all $\alpha \in \acc(C_\beta) \cap \acc(C_\gamma)$, we have
  $C_\beta \cap \alpha = C_\gamma \cap \alpha$.
\end{enumerate}
\end{claim}

\begin{proof}
(1) By the definition of $\bb{P}$ and the fact that $G$ is a filter,
  it follows that $C_\beta$ is a subset of $\beta$, closed in its supremum, such that every proper initial
  segment of $C_\beta$ has size $< \kappa$. It thus suffices to verify that $C_\beta$ is unbounded in $\beta$.
  To this end, fix $\alpha < \beta$. Since $G$ meets $\mathcal{D}$ everywhere, we can find $p \in G \cap D_{\{\alpha, \beta, \beta + 1\}}$.
  Since $\cf(\beta) = \kappa$ and $|x_p| < \kappa$, we have $\beta \in \nacc(x_p)$. Therefore, since $p \in \bb{Q}$,
  we have $\max(f_p(\beta)) = \max(x_p \cap \beta) \geq \alpha$, so $C_\beta \cap [\alpha, \beta) \not= \emptyset$.

(2) Given $\alpha\in\acc(C_\beta)\cap\acc(C_\gamma)$, we fix $p \in G \cap D_{\{\alpha, \beta, \gamma\}}$. As in the previous case, we have $\max(f_p(\beta))\ge\alpha$ and
  $\max(f_p(\gamma)) \geq \alpha$. Consequently, $C_\beta \cap \alpha = f_p(\beta) \cap \alpha$ and $C_\gamma
  \cap \alpha = f_p(\gamma) \cap \alpha$. By the definition of $\bb{P}$, it then follows that $C_\beta \cap \alpha =
  f_p(\alpha) = C_\gamma \cap \alpha$.
\end{proof}

Let $\Gamma := E^{\kappa^+}_\kappa \cup\bigcup\{\acc(C_\beta)\mid \beta\in E^{\kappa^+}_\kappa\}$.
For each $\alpha \in \Gamma \setminus E^{\kappa^+}_\kappa$, find $\beta \in E^{\kappa^+}_\kappa$ such that
$\alpha \in \acc(C_\beta)$, and let $C_\alpha := C_\beta \cap \alpha$. By the preceding Claim,
this is independent of the choice of $\beta$. It follows that $\langle C_\alpha \mid \alpha \in \Gamma \rangle$
is a $\square^B_\kappa$-sequence, thus completing the proof of Theorem~\ref{recovering_thm}.

\subsection{Strong stationary coding sets} In \cite{ShSt:167}, Shelah and Stanley derive a stationary coding set from the existence of a $(\kappa,1)$-morass with built-in $\diamondsuit$.
Specifically, they obtain a stationary subset $\mathcal S$ of $[\kappa^+]^{<\kappa}$ on which the map $x\mapsto\sup(x)$ is one-to-one.
By Theorem~\ref{recovering_thm} and the next proposition, this also follows from the forcing axiom $\fa{\kappa}$.

\begin{prop}[folklore] If $\square^B_\kappa$ holds, then there exists a stationary subset of $[\kappa^+]^{<\kappa}$
on which the map $x\mapsto\sup(x)$ is one-to-one.
\end{prop}
\begin{proof} Let $\langle C_\beta\mid \beta\in\Gamma\rangle$ be a $\square^B_\kappa$-sequence.
Enlarge it to a sequence $\vec{C} = \langle C_\beta \mid \beta <\kappa \rangle$
by letting, for all limit $\beta\in\kappa\setminus\Gamma$, $C_\beta$ be an arbitrary club in $\beta$ of order type $\cf(\beta)$,
and letting $C_{\beta+1}:=\{\beta\}$ for all $\beta<\kappa$.

Let $\rho_1^{\vec C}:[\kappa^+]^2\rightarrow\kappa$ denote the associated \emph{maximal weight} function from \cite[\S6.2]{MR2355670}.
For each $\beta<\kappa^+$, let $\rho_{1\beta}:\beta\rightarrow\kappa$ denote the fiber map $\rho_1^{\vec C}(\cdot,\beta)$.
Note that:
\begin{itemize}
\item for all $\beta<\kappa^+$, $\rho_{1\beta}[C_\beta]=\otp(C_\beta)$;
\item for all $\beta<\kappa^+$, $\rho_{1\beta}$ is $(<\kappa)$-to-$1$;
\item for all $\beta\in\Gamma$ and $\alpha\in\acc(C_\beta)$, we have $\rho_{1\alpha}\s\rho_{1\beta}$.
\end{itemize}

In particular, for every $\beta\in E^{\kappa^+}_{<\kappa}$, we have that $$x_\beta:=(\rho_{1\beta})^{-1}[\otp(C_\beta)]$$
is a cofinal subset of $\beta$ of size $<\kappa$.
Thus, we are left with proving the following.
\begin{claim} $\{ x_\beta\mid \beta\in  E^{\kappa^+}_{<\kappa}\}$ is stationary in $[\kappa^+]^{<\kappa}$.
\end{claim}
\begin{proof} Given a function $f:[\kappa^+]^{<\omega}\rightarrow\kappa^+$,
let us fix some $\gamma\in E^{\kappa^+}_\kappa$ such that $f``[\gamma]^{<\omega}\s\gamma$.
Define $g:\kappa\rightarrow\kappa$ by letting, for all $\varepsilon<\kappa$,
$$g(\varepsilon):=\sup(\rho_{1\gamma}``f``[\rho_{1\gamma}^{-1}[\varepsilon]]^{<\omega}).$$

Fix $\epsilon\in\acc(\kappa)$ such that $g[\epsilon]\s\epsilon$. Put $\beta:=C_\gamma(\epsilon)$, so that $\otp(C_\beta)=\epsilon$ and $\rho_{1\beta}\s\rho_{1\gamma}$.
To see that $f``[x_\beta]^{<\omega}\s x_\beta$,
let  $\{\alpha_i\mid i<n\}\in[x_\beta]^{<\omega}$ be arbitrary.
Since $x_\beta=(\rho_{1\beta})^{-1}[\epsilon]$ and $\rho_{1\beta}\s\rho_{1\gamma}$, we have $$\{\rho_{1\gamma}(\alpha_i)\mid i<n\}=\{ \rho_{1\beta}(\alpha_i)\mid i<n\}\in [\epsilon]^{<\omega}.$$
Fix a large enough $\varepsilon<\epsilon$ such that $\{\alpha_i\mid i<n\}\in[\rho_{1\gamma}^{-1}[\varepsilon]]^{<\omega}$.
Since $g(\varepsilon)<\epsilon$, we then have $f(\{\alpha_i\mid i<n\})\in x_\beta$.
\end{proof}
This completes the proof.
\end{proof}

Note that, by \cite[\S3]{paper08},
strong stationary coding sets can be seen as a $\gch$-free version of $\diamondsuit$.
For more information on stationary coding sets, see \cite{MR763904}.

\section{Super-Souslin trees} \label{tree_sect}

Throughout this section, $\lambda$ denotes an arbitrary cardinal.

\smallskip

The notion of a $\lambda^{++}$-super-Souslin tree was isolated by Shelah in response to work by Laver
on trees with ascent paths.
Ascent paths provide obstacles to a tree being special;
super-Souslin trees are designed to present a similar obstacle that entails the existence
not only of a non-special tree but of a Souslin one.
In Subsection~\ref{super-intro}, we provide, as a means of helping to motivate and provide context for the definition of super-Souslin trees, some remarks on the connection between these notions.
In Subsection~\ref{prfthmc}, we provide a proof of Theorem~C:
\begin{thmc} \label{tree_thm}
  Suppose that  $\fa{\lambda^+}$ holds. Then there exists a $\lambda^+$-complete
  $\lambda^{++}$-super-Souslin tree.
\end{thmc}

\subsection{Introduction to super-Souslin trees}\label{super-intro}

A tree $(T,<_T)$ is said to be a \emph{$\kappa$-tree} if for every $\alpha<\kappa$,
$T_\alpha$ is a nonempty set of size $<\kappa$ and $T_{\kappa}=\emptyset$.
The tree is said to be \emph{splitting} if every node in the tree admits at least two immediate successors.
It is said to be \emph{normal} if, for all $\alpha < \beta < \kappa$ and all $u \in T_\alpha$,
there is $v \in T_\beta$ such that $u <_T v$.
It is said to be \emph{Hausdorff} if for all limit $\alpha<\kappa$ and all $u,v\in T_\alpha$,
the equality $u_\downarrow=v_\downarrow$ implies $u=v$.
For convenience, we will \emph{not} require that a tree be Hausdorff.
Note, however, that any splitting (resp. normal) tree $(T,<_T)$ can easily be turned into a splitting (resp. normal) Hausdorff tree $(T',<_{T'})$ by shifting all
levels $T_\alpha$ to be $T'_{\alpha+1}$ and, for limit $\alpha < \kappa$, letting $T'_\alpha$ consist of unique limits
of all branches through $\bigcup_{\beta < \alpha} T_\beta$ that are continued in $T_\alpha$.

\begin{defn} Let $\theta$ be an arbitrary cardinal. For each $\alpha<\kappa$,
let $T^\theta_\alpha$ denote the collection of all injections $a:\theta\rightarrow T_\alpha$.
Let $T^\theta$ denote $\bigcup_{\alpha < \kappa} T^\theta_\alpha$.
\end{defn}
An element of $T^\theta$ will be referred to as a \emph{$\theta$-level sequence} from $T$ (or, simply, a \emph{level sequence} from $T$).
For $a,b \in T^\theta$, we abuse notation and write $a <_T b$ iff, for all $i < \theta$, $a(i) <_T b(i)$.
Likewise, $a \le_T b$ iff, for all $i < \theta$, $a(i) \le_T b(i)$.

\begin{defn} $[T^\theta]^2 := \{(a,b)\in T^\theta\times T^\theta\mid a <_T b\}$.
\end{defn}

\begin{defn}[Shelah, \cite{shelah_stanley}]\label{superdef}
  A $\lambda^{++}$-super-Souslin tree is a normal, splitting $\lambda^{++}$-tree
  $(T,<_T)$ for which there exists a function $F:[T^\lambda]^2 \rightarrow \lambda^+$ satisfying the following condition:
  for all $a,b,c \in T^\lambda$ with $a <_T b,c$, if $F(a,b) = F(a,c)$, then there is $i < \lambda$
  such that $b(i)$ and $c(i)$ are $<_T$-comparable.
\end{defn}

\begin{fact}[Shelah, \cite{shelah_stanley}] Suppose $(T,<_T)$ is a $\lambda^{++}$-super-Souslin tree.
    If $W$ is an outer model of $V$ with the same $\mathcal{P}(\lambda)$ and $\lambda^{++}$,
    then, in $W$, there exists some $x\in T$ such that $(x^\uparrow,<_T)$ is a $\lambda^{++}$-Souslin    tree.
\end{fact}

The next lemma shows that the two-dimensional function $F$ witnessing that a tree $(T,<_T)$ is $\lambda^{++}$-super-Souslin
cannot be replaced by a one-dimensional function.

\begin{lemma} Suppose that $(T,<_T)$ is a normal, splitting $\kappa$-tree, and $\theta,\mu$ are cardinals $<\kappa$ (e.g., $\kappa=\lambda^{++}$, $\mu=\lambda^+$, and $\theta=\lambda$.)
  There exists no function $F:T^\theta\rightarrow\mu$ such that,
  for every $a,b \in T^\theta$, if $F(a) = F(b)$, then there is $i < \theta$ such that $a(i)$ and $b(i)$ are
  $<_T$-comparable.
\end{lemma}
\begin{proof}
  Suppose for sake of contradiction that there is such a function $F$. We first argue that $(T,<_T)$ is a $\kappa$-Souslin tree. Furthermore:
  \begin{subclaim}
    Suppose $W$ is an outer model of $V$ in which $\kappa$ is not collapsed.
    Then $(T, <_T)$ is a $\kappa$-Souslin    tree in $W$.
  \end{subclaim}
  \begin{proof} Work in $V$. As the proof of Claim~A.7.1 of \cite{paper20} makes clear,
  the fact that $(T,<_T)$ is normal and splitting implies that for every $u\in T$, we may find some $a_u\in T^\theta$ such that $u<_T a_u(i)$ for all $i<\theta$.
  Next, let us work in $W$, where $W$ is an outer model of $V$ in which $\kappa$ is not collapsed.
  Since $(T, <_T)$ is a splitting $\kappa$-tree, to show that $(T, <_T)$ is $\kappa$-Souslin, it suffices to show that it has no antichains of size $\kappa$.
    Towards a contradiction, suppose that $U:=\{ u_\alpha\mid\alpha<\kappa\}$ is an antichain.
    While it is possible that $U\in W\setminus V$, we nevertheless have $\{ a_{u_\alpha}\mid \alpha<\kappa\}\s V$.
    Since $\kappa$ is not collapsed, we may find ordinals $\alpha<\beta<\kappa$ such that $F(a_{u_\alpha})=F(a_{u_\beta})$.
    Pick $i<\theta$ such that $a_{u_{\alpha}}(i)$ and $a_{u_{\beta}}(i)$ are $<_T$-comparable. Then $u_\alpha$ and $u_\beta$ are $<_T$-comparable.
    This is a contradiction.
  \end{proof}

  Now force over $V$ with the forcing notion $\mathbb{P} := (T, >_T)$ (i.e., the order of $\mathbb{P}$ is the reverse of the tree order).
  As $(T,<_T)$ is a $\kappa$-Souslin tree in $V$, we have that $\mathbb{P}$ has the $\kappa$-c.c.\ and does not collapse $\kappa$.
  Therefore, the preceding claim implies that $(T, <_T)$ is a $\kappa$-Souslin tree in $V^{\mathbb{P}}$,
  contradicting the fact that $\mathbb{P}$ adds a cofinal branch through $(T, <_T)$.
\end{proof}

The next lemma shows that the range of the function $F$ witnessing that a tree $(T,<_T)$ is $\lambda^{++}$-super-Souslin
cannot be smaller than $\lambda^{+}$. In particular, there is no straightforward generalization of the notion of
super-Souslin tree to inaccessible cardinals.

\begin{lemma} Suppose that $(T,<_T)$ is a normal, splitting $\kappa$-tree, and $\theta,\mu$ are cardinals $<\kappa$.
  If $\mu^+<\kappa$, then there exists no function $F:[T^\theta]^2\rightarrow\mu$ such that,
    for all $a,b,c \in T^\theta$ with $a <_T b,c$, if $F(a,b) = F(a,c)$, then there is $i < \theta$
  such that $b(i)$ and $c(i)$ are $<_T$-comparable.
\end{lemma}
\begin{proof} Suppose that $F$ is a counterexample.
Fix an arbitrary $a\in T^\theta$.
As the proof of Claim~A.7.1 of \cite{paper20} makes clear,
  the fact that $(T,<_T)$ is normal and splitting implies that there exists some large enough $\beta<\kappa$
  and an injection $b:\mu^+\times \theta\rightarrow T_\beta$ such that for all $\eta<\mu^+$ and all $i<\theta$, $a(i)<_T b(\eta,i)$.
  For each $\eta<\mu^+$, define $b_\eta:\theta\rightarrow T_\beta$ by stipulating $b_\eta(i):=b(\eta,i)$.
  Now, find $\eta<\zeta<\mu^+$ such that $F(a,b_\eta)=F(a,b_\zeta)$.
  Then there must exist some $i<\theta$ such that $b_\eta(i)$ and $b_\zeta(i)$ are $<_T$-comparable, contradicting the fact that
  $b_\eta(i)$ and $b_\zeta(i)$ are two distinct elements of $T_\beta$.
\end{proof}

Now, we move on to deal with the notion of an ascent path.

\begin{defn}[Laver] Suppose that $\theta$ is a cardinal $<\kappa$ and $\mathcal F$ is a family satisfying $\theta\in\mathcal F\s\mathcal P(\theta)$.
An \emph{$\mathcal F$-ascent path} through a $\kappa$-tree $(T,<_T)$ is
a sequence $\vec f =\langle f_\alpha\mid \alpha<\kappa\rangle$ such that for all $\alpha<\beta<\kappa$:
\begin{enumerate}
\item $f_\alpha$ is a function from $\theta$ to $T_\alpha$;
\item $\{ i<\theta\mid f_\alpha(i) <_T f_\beta(i) \}\in \mathcal F$.
\end{enumerate}
\end{defn}

\begin{defn} For every cardinal $\theta$, write
$\mathcal F^{\fin}_\theta:=\{ Z\s\theta\mid |\theta\setminus Z|<\omega\}$,
$\mathcal F^{\bd}_\theta:=\{ Z\s\theta\mid \sup(\theta\setminus Z)<\theta\}$,
and $\mathcal F_\theta:=\mathcal P(\theta)\setminus\{\emptyset\}$.
\end{defn}

By \cite{MR964870}, if $(T,<_T)$ is a special $\lambda^+$-tree
that admits an $\mathcal F^{\bd}_\theta$-ascent path, then $\cf(\theta)=\cf(\lambda)$.
By \cite{MR2965421}, if $\lambda$ is regular and $(T,<_T)$ is a special $\lambda^+$-tree
that admits an $\mathcal F_\theta$-ascent path, then $\theta=\lambda$.
A construction of a special $\lambda^+$-tree with an $\mathcal F^{\bd}_{\cf(\lambda)}$-ascent path may be found in \cite{lh_trees}.
Constructions of $\kappa$-Souslin trees with $\mathcal F^{\fin}_\theta$-ascent paths may be found in \cite{paper20}.

\begin{prop}[folklore]
Any $\lambda^{++}$-super-Souslin tree $(T,<_T)$ admits an $\mathcal F_\lambda$-ascent path.
\end{prop}
\begin{proof} Suppose $(T,<_T)$ is a $\lambda^{++}$-super-Souslin tree with a witnessing map $F:[T^\lambda]^2\rightarrow\lambda^+$.
Fix an arbitrary $a\in T^\lambda$. Let $\epsilon$ be such that $a\in T^\lambda_\epsilon$.
By normality of $(T,<_T)$, for each $\beta\in \lambda^{++}\setminus\epsilon$, we may fix $a_\beta\in T_\beta^\lambda$ with $a\le_T a_\beta$.
Pick a cofinal subset $B\s\lambda^{++}\setminus\epsilon$ on which the map $\beta\mapsto F(a,a_\beta)$ is constant.
Then $\langle a_\beta\mid \beta\in B\rangle$ induces an $\mathcal F_\theta$-ascent path $\vec f=\langle f_\alpha\mid\alpha<\kappa\rangle$, as follows.
For every $\alpha<\lambda^{++}$, let $\beta(\alpha):=\min(B\setminus\alpha)$, and define $f_\alpha:\lambda\rightarrow T_\alpha$
by letting $f_\alpha(i)$ be the unique element of $T_\alpha$ which is $\le_T a_{\beta(\alpha)}(i)$.
\end{proof}

Aiming for an $\mathcal F_\lambda^{\bd}$-ascent path,
one may want to strengthen Definition~\ref{superdef}
to assert that   for all $a,b,c \in T^\lambda$ with $a <_T b,c$, if $F(a,b) = F(a,c)$, then
$I(b,c):=\{i < \lambda \mid b(i)\text{ and }c(i)\text{ are }{<_T}\text{-comparable}\}$ is in $\mathcal F_\lambda^{\bd}$.
However, this is impossible:

\begin{lemma}
  Suppose $(T,<_T)$ is a normal, splitting $\lambda^{++}$-tree, $F:[T^\lambda]^2 \rightarrow \lambda^+$, and $\mathcal F$ is a proper filter
  on $\lambda$. Then there are $(a,b),(a,c) \in [T^\lambda]^2$ with $F(a,b) = F(a,c)$ such that $I(b,c)\notin\mathcal F$.
\end{lemma}

\begin{proof}
  Towards a contradiction, suppose that for all $(a,b),(a,c) \in [T^\lambda]^2$ with $F(a,b) = F(a,c)$, we have  $I(b,c)\in \mathcal F$.
  For all $a\in T^\lambda$ and $\eta<\lambda^+$, let
  $U_a:=\{ b\in T^\lambda\mid a\le _T b\}$ and $U^\eta_a:=\{ b\in U_a\mid F(a,b)=\eta\}$.
  Now, fix some $a\in T^\lambda$ arbitrarily,
  and, for every $\eta < \lambda^+$, let $U^\eta:=\{ b\in U_a\mid U_b\cap U_a^\eta\neq\emptyset\}$
  be the downward closure of $U^\eta_a$ within $U_a$.

  \begin{subclaim} \label{comparable_claim}
    Suppose that $\eta<\lambda^+$ and $b,c \in U^\eta$. Then $I(b,c)\in\mathcal F$.
  \end{subclaim}
  \begin{proof}
    Pick $b'\in U_b\cap U^\eta_a$ and $c'\in U_c\cap U^\eta_a$.
    Since $F(a,b') = \eta = F(a,c')$, and by assumption, we have that $I(b',c')\in\mathcal F$.

    Let $\beta, \beta', \gamma, \gamma'$ be such that
    $b \in T_\beta^\lambda$, $b' \in T_{\beta'}^\lambda$, $c \in T_\gamma^\lambda$, and
    $c' \in T_{\gamma'}^\lambda$.  Without loss of generality,  $\beta' \leq \gamma'$.
    Now, there are two relevant configurations of the other ordinals to consider.

    \textbf{Case 1: $\beta \leq \beta' < \gamma$.} In this case, for all $i \in I(b',c')$, we have
    $b(i) \leq_T b'(i)$ and $b'(i), c(i) \leq_T c'(i)$, so
    $b(i)$ and $c(i)$ are $<_T$-comparable.

    \textbf{Case 2: $\beta, \gamma \leq \beta'$.} In this case, for all $i \in I(b',c')$, we have
    $b(i), c(i) \leq_T b'(i)$ and again, $b(i)$ and $c(i)$ are $<_T$-comparable.
  \end{proof}

  For any two distinct ordinals $\eta,\zeta$ below $\lambda^+$, let $\delta_{\eta, \zeta}$ denote the least ordinal $\delta$ below $\lambda^{++}$
  such that there are $b \in U^\eta \cap T_\delta^\lambda$ and $c \in U^\zeta \cap T_\delta^\lambda$
  for which $I(b,c)=\emptyset$, if such an ordinal exists; otherwise, leave $\delta_{\eta, \zeta}$
  undefined.

  \begin{subclaim} \label{divergence_claim}
    Suppose $\delta_{\eta, \zeta}$ is defined.
    Then $U^\eta \cap U^\zeta\s\bigcup\{T_\beta^\lambda\mid \beta <\delta_{\eta, \zeta}\}$.
  \end{subclaim}

  \begin{proof} Towards a contradiction, suppose that $d\in U^\eta\cap U^\zeta\cap T^\lambda_\beta$ for some $\beta\ge\delta_{\eta, \zeta}$.
    Since $U^\eta$ and $U^\zeta$ are downward closed, we may simply assume that $\beta=\delta_{\eta, \zeta}$.

    Using the fact that $\beta=\delta_{\eta, \zeta}$, fix $b \in U^\eta \cap T_\beta^\lambda$ and $c \in U^\zeta \cap T_\beta^\lambda$
    such that $I(b,c)=\emptyset$. By Claim~\ref{comparable_claim},
    since $b,d \in U^\eta$ and $c,d\in U^\zeta$, we have that $I(b,d)$ and $I(c,d)$ are in $\mathcal F$.
    In particular, $I(b,d)\cap I(c,d)\neq\emptyset$, contradicting the fact that $I(b,c)=\emptyset$.
  \end{proof}

  As $\lambda^+<\lambda^{++}$, let $\beta < \lambda^{++}$ be large enough so that, if $\eta,\zeta$ are two distinct ordinals below $\lambda^+$ and $\delta_{\eta, \zeta}$ is
  defined, then $\delta_{\eta, \zeta} < \beta$. By increasing $\beta$ if necessary, we may assume that $U_a\cap T_\beta^\lambda\neq\emptyset$.
  Fix $d \in U_a\cap T_\beta^\lambda$. By the fact that $(T,<_T)$ is splitting, for each
  $i < \lambda$ we may fix $e_0(i) \neq e_1(i)$, both in $T_{\beta+1}$, with $d(i) <_T e_0(i), e_1(i)$.
  Let $\eta := F(a,e_0)$ and $\zeta := F(a,e_1)$. Clearly, $I(e_0,e_1)=\emptyset$,
  so that  $\delta_{\eta, \zeta}$ is defined. So, by our choice of
  $\beta$, we have $\delta_{\eta, \zeta} < \beta$. However, since
  $d <_T e_0, e_1$, we have $d \in U^\eta \cap U^\zeta\cap T^\lambda_\beta$, contradicting Claim \ref{divergence_claim}.
\end{proof}

\subsection{Proof of Theorem~C}\label{prfthmc}

The rest of this section is devoted to proving Theorem~C.
We will define a poset $(\bb{P}, \leq_{\bb{P}}, \bb{Q})
\in \mathcal{P}_{\lambda^+}$ and a collection $\{ \mathcal D_i\mid i<\lambda^{+}\}$ of sharply dense systems such that
any filter that meets each $\mathcal D_i$ everywhere gives rise to a $\lambda^+$-complete $\lambda^{++}$-super-Souslin tree.
We intend to construct a tree $(T, <_T)$ with underlying set $\lambda^{++} \times \lambda^+$, such that, furthermore,
 $T_\alpha = \{\alpha\} \times \lambda^+$ for all $\alpha<\lambda^{++}$. We start by defining $\bb{P}$.

\begin{defn} \label{p_defn}
  $\bb{P}$ consists of all quintuples $p = (x, <^0, t, <^1, f)$ satisfying
  the following requirements.
  \begin{enumerate}
    \item $x \in [\lambda^{++}]^{<\lambda^+}$.
    \item $<^0$ is a partial ordering on $x$ such that for all $\beta \in x$, we have that
      $$\pred^0_p(\beta) := \{\alpha \in x \mid \alpha <^0 \beta\}$$
          is a closed subset of $\beta$ which is well-ordered by $<^0$.
    \item\label{p_defn3} $t \in [x \times \lambda^+]^{<\lambda^+}$. In a slight abuse of notation, and anticipating the generic object, for all
      $\alpha \in x$, we let $t_\alpha$ denote $t \cap (\{\alpha\} \times \lambda^+)$,
      and we let $t_\alpha^\lambda$ denote the set of injective
      functions from $\lambda$ to $t_\alpha$.
      For each $a$ in $t^\lambda := \bigcup_{\alpha \in x} t_\alpha^\lambda$,
      we let $\lht(a)$ denote the unique ordinal $\alpha$ such that $a\in t_\alpha^\lambda$.

    \item\label{p_defn4} $<^1$ is a tree order on $t$ such that, for all $\beta \in x$ and all $v \in t_\beta$, letting
      $\pred^1_p(v) := \{u \in t \mid u <^1 v\}$, we have that
      \[
        \{\alpha \in x \mid \pred^1_p(v) \cap t_\alpha \neq \emptyset\}
        = \pred^0_p(\beta).
      \]
      Let $[t^\lambda]^2 := \{(a,b) \mid a,b \in t^\lambda, a <^1 b\}$,
      where for $a,b \in t^\lambda$, we write $a <^1 b$ iff  $a(i) <^1 b(i)$ for all $i<\lambda$.
    \item $f$ is a partial function from $[t^\lambda]^2$ to $[\lambda^+]^{<\lambda^+} \setminus \{\emptyset\}$, and $|f|\leq \lambda$.
    \item\label{p_defn6} Suppose that $(a,b), (a,c) \in \dom(f)$.
   If  $f(a,b) \cap f(a,c) \neq \emptyset$ and $\lht(b) \leq^0 \lht(c)$, then
$|\{i < \lambda \mid b(i) \leq^1 c(i)\}| = \lambda$.
    \item\label{p_defn7} For all $(a,c) \in \dom(f)$ and all $b \in t^\lambda$ such that $a <^1 b <^1 c$, we have
      $(a,b) \in \dom(f)$ and $f(a,b) \supseteq f(a,c)$.
  \end{enumerate}
\end{defn}
  The coordinates of a condition $p \in \bb{P}$ will often be identified as $x_p, <^0_p, t_p,
  <^1_p$, and $f_p$, respectively.
\begin{defn} For all $p,q \in \bb{P}$, we let $q \leq_{\bb{P}} p$ iff:
  \begin{itemize}
    \item $x_q \supseteq x_p$;
    \item $<^0_q \supseteq <^0_p$;
    \item $t_q \supseteq t_p$;
    \item $<^1_q \supseteq <^1_p$;
    \item $\dom(f_q) \supseteq \dom(f_p)$;
    \item for all $(a,b) \in \dom(f_p)$, we have $f_q(a,b) \supseteq f_p(a,b)$.
  \end{itemize}
\end{defn}

\begin{defn} \label{q_defn}
  $\bb{Q}$ is the set of all conditions $p\in \bb{P}$   such that:
\begin{enumerate}
\item    $x_p=\cl(x_p)$;
\item $<^0_p$ is the usual  ordinal ordering on $x_p$.
\end{enumerate}
\end{defn}

We now show that $(\bb{P}, \leq_{\bb{P}}, \bb{Q})$ is in $\mathcal{P}_{\lambda^+}$.
For $p \in \bb{P}$, $x_p$ is the realm of $p$. We next describe how \cois act on $\bb{P}$.
In order to make it easier to refer to and manipulate level sequences in our conditions, we introduce
the following notation.

\begin{notation}
  By $\fa{\lambda^+}$ and Corollary~\ref{cardinalarithmetic}, $\ch_\lambda$ holds, and we can let $\langle \sigma_\delta \mid \delta < \lambda^+ \rangle$ injectively enumerate ${^\lambda} \lambda^+$. For all
  $\alpha < \lambda^{++}$ and $\delta < \lambda^+$, let $a_{\alpha, \delta}:\lambda \rightarrow \{\alpha\}\times \lambda^+$ be defined by stipulating $a_{\alpha, \delta}(i) := (\alpha, \sigma_\delta(i))$.
  Note that  every level sequence in our desired tree $(T,<_T)$ will be of the form $a_{\alpha, \delta}$ for a unique pair $(\alpha,\delta)\in \lambda^{++}\times\lambda^+$.
\end{notation}

\begin{defn}
  Suppose that $\pi$ is a \coi from a subset of $\lambda^{++}$ to
  $\lambda^{++}$. For each  $p \in \bb{P}_{\dom(\pi)}$,  we define $\pi.p$ to be the condition
  $(x, <^0, t, <^1, f) \in \bb{P}$ such that:
  \begin{enumerate}
    \item $x = \pi``x_p$;
    \item $<^0=\{(\pi(\alpha),\pi(\beta))\mid (\alpha,\beta)\in <^0_p\}$;
    \item $t=\{(\pi(\alpha),\eta)\mid (\alpha,\eta)\in t_p\}$;
    \item $<^1=\{((\pi(\alpha),\eta),(\pi(\beta),\zeta))\mid ((\alpha,\eta),(\beta,\zeta))\in <^1_p\}$;
\item $f=\{ ((a_{\pi(\alpha),\delta},a_{\pi(\beta),\epsilon}),z)\mid ((a_{\alpha,\delta},a_{\beta,\epsilon}),z)\in f_p\}$.
  \end{enumerate}
\end{defn}

Finally, we describe the restriction operation.

\begin{defn}
  Suppose that $p \in \bb{P}$ and $\alpha < \lambda^{++}$. Then
  $p \rest \alpha$ is the condition $(x, <^0, t, <^1, f)$ such that:
  \begin{itemize}
    \item $x = x_p \cap \alpha$;
    \item $<^0 = <^0_p \cap x^2$;
    \item $t = t_p \cap (\alpha \times \lambda^+)$;
    \item $<^1 = <^1_p \cap t^2$;
    \item $f = \{((a,b),z)\in f_p\mid (a,b)\in  [t^\lambda]^2\}$.
  \end{itemize}
\end{defn}

With these definitions, it follows easily that $(\bb{P}, \leq_{\bb{P}}, \bb{Q})$ satisfies
Clauses \eqref{c1}--\eqref{c12} of Definition~\ref{class_def}. We now verify Clauses \eqref{c9}--\eqref{c11}, in order.

\begin{lemma} \label{order_lemma}
  Suppose $p \in \bb{P}$. Then there is $q \in \bb{Q}$ with $q \leq_{\bb{P}} p$  such that $x_q = \cl(x_p)$.
\end{lemma}

\begin{proof} We need to define $q=(x_q,<^0_q,t_q,<^1_q,f_p)$.
  Of course, we let $x_q := \cl(x_p)$ and let $<^0_q$ be the usual ordinal ordering
  on $x_q$.  Thus, the main task is in finding suitable $t_q$, $<^1_q$ and $f_q$.
  Our strategy is to define the first two and then derive $f_q$ by minimally extending $f_p$ so as
  to satisfy Clause~\eqref{p_defn7} of Definition \ref{p_defn}. To be precise,
  once $t_q$ and $<^1_q$ are determined, we will let
  $$\dom(f_q) := \dom(f_p) \cup
  \{(a,b) \in [t_q^\lambda]^2 \mid \exists c \in t_p^\lambda \left( a <^1_q b <^1_q c \text{ and }
  (a,c) \in \dom(f_p) \right) \}$$
  and, for all $(a,b) \in \dom(f_q)$, we will let
$$f_q(a,b) := f_p(a,b) \cup \bigcup \{f_p(a,c) \mid  (a,c) \in \dom(f_p)\text{ and }a <^1_q b <^1_q c\}.$$

   We now turn to defining $t_q$ and $<^1_q$ to ensure that Clauses \eqref{p_defn4} and \eqref{p_defn6} of Definition~\ref{p_defn} hold.
By our intended definition of $f_q$ and $<^0_q$,
 these clauses dictate that, for all $\beta\in x_q$ and $\alpha\in x_q\cap(\beta+1)$:
\begin{itemize}
\item[($\ref{p_defn4}'$)] for all $v\in(t_q)_\beta$, there is some $u\in(t_q)_\alpha$ with $u\le^1_q v$;
\item[($\ref{p_defn6}'$)] for all $(a,b), (a,c) \in \dom(f_p)$ with $f_p(a,b) \cap f_p(a,c) \neq \emptyset$,
if $b \in (t_p)_\alpha^\lambda$ and $c \in (t_p)_\beta^\lambda$, then  $|\{i < \lambda \mid b(i) \leq^1_q c(i)\}| = \lambda$.
\end{itemize}

  In order to satisfy Clause~($\ref{p_defn4}'$), it is possible that we will have to add new nodes to $t_q$, i.e., that $t_q \setminus t_p \neq \emptyset$.
  However, we will do so in such a way that each element of $t_q \setminus t_p$ will be a $<^1_q$-predecessor of an
  element of $t_p$. Consequently, to define $t_q$ and $<^1_q$, it suffices to specify $\pred^1_q(v)$ for all  $v \in t_p$.

  Now, by recursion on $\beta\in x_p$, we define
$\pred^1_q(v)$ for all $v\in(t_p)_\beta$ in a way that ensures that Clauses
($\ref{p_defn4}'$) and ($\ref{p_defn6}'$) hold for all $\alpha\in x_q\cap(\beta+1)$.
  Suppose that $\beta \in x_p$ and, for all $\alpha \in x_p \cap \beta$, we have
  specified $\pred^1_q(u)$ for every $u \in (t_p)_\alpha$.
  Let $t_{<\beta}$ denote the underlying set of the tree we have
  defined thus far, i.e., $\bigcup_{\alpha\in x_p\cap\beta}\bigcup_{u \in (t_p)_{\alpha}}
  (\pred^1_q(u)\cup\{u\})$.

  If $\pred^0_p(\beta) = \emptyset$ and $v \in (t_p)_\beta$, then let $B$ be a maximal branch through $t_{<\beta}$.
  It might be the case that $B$ is bounded below $\beta$, i.e., there is $\gamma \in x_q \cap \beta$ with
  $B \cap (\{\gamma\} \times \lambda^+) = \emptyset$.  If this is the case, then, for each
  such $\gamma$, add a new element from $\{\gamma\} \times \lambda^+$ to $t_q$
  and require that these new elements, together with $B$, form a branch whose levels are unbounded in $x_q \cap \beta$. Let this
  unbounded branch be denoted by $B^*$, and set $\pred^1_q(v) := B^*$.

  If $\pred^0_p(\beta)$ is unbounded in $\beta$,  then, for all $v \in (t_p)_\beta$, we are obliged to let $\pred^1_q(v)$
  be precisely $\bigcup_{u \in \pred^1_p(v)} (\pred^1_q(u)\cup\{u\})$.

  It remains to consider the case in which $\pred^0_p(\beta)$ is nonempty and bounded in $\beta$.
  Put $\beta' := \sup(\pred^0_p(\beta))$.
  Since $\pred^0_p(\beta)$ is a closed, nonempty subset of $\beta$, we have $\beta'\in x_p$.
  If there is no $\gamma \in x_p$ with $\beta' < \gamma < \beta$,
  then, for all $v \in (t_p)_\beta$, we are again obliged to let $\pred^1_q(v):=\bigcup_{u \in \pred^1_p(v)} (\pred^1_q(u)\cup\{u\})$.
  Thus, from now on, suppose that $x_p \cap (\beta', \beta) \neq \emptyset$.

  Let $\langle (a_\ell, b_\ell, c_\ell) \mid \ell < \lambda \rangle$ enumerate all triples
  $(a,b,c)$ such that:
  \begin{itemize}
    \item $(a,b), (a,c) \in \dom(f_p)$;
    \item $f_p(a,b) \cap f_p(a,c) \neq \emptyset$;
    \item $c \in (t_p)_\beta^\lambda$ and there is $\alpha \in x_p \cap (\beta', \beta)$
      such that $b \in (t_p)_\alpha^\lambda$.
  \end{itemize}
  Moreover, assume that each such triple is enumerated as $(a_\ell, b_\ell, c_\ell)$ for
  $\lambda$-many $\ell < \lambda$. (If there are no such triples, then simply define
  $\pred^1_q(v)$ arbitrarily for each $v \in (t_p)_\beta$ subject to the constraint
  $\pred^1_q(v) \supseteq \pred^1_p(v)$.)

  Now, by recursion on $\ell < \lambda$, we choose nodes $v_\ell \in (t_p)_\beta$ and
  specify $\pred^1_q(v_\ell)$. Suppose that $\ell < \lambda$ and we have chosen
  $v_{\ell'}$ and $\pred^1_q(v_{\ell'})$ for all $\ell' < \ell$. Consider the triple
  $(a_\ell, b_\ell, c_\ell)$.

  Suppose first that $a_\ell \in (t_p)_{\beta'}$. We have that, for all $i < \lambda$,
  $a_\ell(i) <^1_p b_\ell(i), c_\ell(i)$. In particular, since $\beta' = \max(\pred^0_p(\beta))$,
  we have, for all $i < \lambda$, $\pred^1_q(b_\ell(i)) \supseteq \pred^1_p(b_\ell(i)) \supseteq \pred^1_p(c_\ell(i))$.
  Choose $i < \lambda$ such that $c_\ell(i) \notin \{v_{\ell'} \mid \ell' < \ell\}$, set $v_\ell := c_\ell(i)$, and let $B$ be a
  maximal branch through $t_{<\beta}$ with $b_\ell(i) \in B$. As in the case in which $\pred^0_p(\beta) = \emptyset$,
  extend $B$, by adding nodes if necessary, to a branch $B^*$ whose levels are unbounded in $x_q \cap \beta$,
  and set $\pred^1_q(v_\ell) := B^*$.

  Suppose next that $a_\ell \in (t_p)_{<\beta'}$. Let $c' \in (t_p)_{\beta'}^\lambda$ be the unique level sequence
  such that $a_\ell <^1_p c' <^1_p c_\ell$. Since $p \in \bb{P}$, we have $(a_\ell, c') \in \dom(f_p)$
  and $f_p(a_\ell, c') \supseteq f_p(a_\ell, c_\ell)$. In particular, $f_p(a_\ell, c') \cap f_p(a_\ell, b_\ell) \neq
  \emptyset$, so, by our inductive hypothesis, we know that, for $\lambda$-many $i < \lambda$, we have
  $c'(i) <^1_q b_\ell(i)$. Choose such an $i$ with $c_\ell(i) \not\in \{v_{\ell'} \mid \ell' < \ell\}$ and let
  $v_\ell := c_\ell(i)$. As in the previous case, by adding nodes if necessary, fix a branch $B^*$ whose levels are unbounded in
  $x_q \cap \beta$ with $b_\ell(i) \in B^*$, and set $\pred^1_q(v_\ell) := B^*$.

  At the end of this process, if there are nodes in $(t_p)_\beta \setminus \{v_\ell \mid \ell < \lambda\}$, then
  assign their $<^1_q$-predecessors arbitrarily. We must verify that we have maintained the inductive hypothesis.
  To this end, fix $(a,b,c)$ such that:
  \begin{itemize}
    \item $(a,b), (a,c) \in \dom(f_p)$;
    \item $f_p(a,b) \cap f_p(a,c) \neq \emptyset$;
    \item $c \in (t_p)_\beta^\lambda$ and there is $\alpha \in x_p \cap \beta$ such that
      $b \in (t_p)_\alpha^\lambda$.
  \end{itemize}

  Suppose first that $\alpha \leq \beta'$. This implies that $a \in (t_p)_{<\beta'}^\lambda$. Therefore, we
  can let $c' \in (t_p)_{\beta'}^\lambda$ be the unique level sequence such that $a <^1_p c' <^1_p c$.
  Then $f_p(a,c') \supseteq f_p(a,c)$, so, by the inductive hypothesis applied at $\beta'$, we have that,
  for $\lambda$-many $i < \lambda$, $b(i) \leq^1_q c'(i) \leq^1_q c(i)$, so we are done.

  Next, suppose $\beta' < \alpha < \beta$. In this case, for $\lambda$-many $\ell < \lambda$, we have $(a,b,c) = (a_\ell,
  b_\ell, c_\ell)$. For each such $\ell$, at stage $\ell$ of the construction, we chose a distinct $i < \lambda$
  and ensured that $b_\ell(i) <^1_q c_\ell(i)$, so, for $\lambda$-many $i < \lambda$, we have $b(i) <^1_q c(i)$,
  as desired.
\end{proof}

\begin{lemma}
  Suppose that $\xi < \lambda^+$ and $\langle q_\eta \mid \eta < \xi \rangle$ is a decreasing sequence from
  $\bb{Q}$. Let $x := \bigcup_{\eta < \xi} x_{q_\eta}$. Suppose that $\alpha < \ssup(x)$ and that
  $r \in \bb{Q}_{\ssup(x \cap \alpha)}$ is a lower bound for $\langle q_\eta \rest \alpha \mid \eta < \xi \rangle$.
  Then there is $q \in \bb{Q}$ such that:
  \begin{itemize}
    \item $q$ is a lower bound for $\langle q_\eta \mid \eta < \xi \rangle$;
    \item $q \rest \ssup(x \cap \alpha) = r$;
    \item $x_q = \cl(x_r \cup x)$.
  \end{itemize}
\end{lemma}

\begin{proof}
  $x_q$ and $<^0_q$ are determined by the requirements of the Lemma. We now specify $t_q, <^1_q$, and $f_q$.
  We must let $q \rest \ssup(x \cap \alpha) = r$, so we only deal with the parts of $t_q$, $<^1_q$, and $f_q$ related to levels at $\ssup(x \cap \alpha)$ or higher.

  Fix $\beta \in x_q \setminus \ssup(x \cap \alpha)$. If $\beta \in x$, then let $(t_q)_\beta :=
  \bigcup_{\eta < \xi} (t_{q_\eta})_\beta$. If $\beta \not\in x$, then let
  $\gamma := \min(x \setminus \beta)$, and let
  $(t_q)_\beta := \{(\beta, \zeta) \mid (\gamma, \zeta) \in \bigcup_{\eta < \xi} (t_{q_\eta})_\gamma\}$.

  We define $<^1_q$ by specifying $\pred^1_q(v)$ for all $v \in t_q$. This is already done for all
  $v \in (t_q)_{<\ssup(x \cap \alpha)}$. We take care of the $v \in (t_q)_{\geq \ssup(x \cap \alpha)}$ by recursion on the $\beta \in x_q$
  such that $v \in (t_q)_\beta$. Thus, suppose $\beta \in x_q \setminus \ssup(x \cap \alpha)$ and we have defined
  $\pred^1_q(u)$ for all $u \in (t_q)_{<\beta}$.

  Suppose first that $\beta \not\in x$, and let $\gamma := \min(x \setminus \beta)$.
  If $v = (\beta, \zeta) \in (t_q)_\beta$, then let $v' := (\gamma, \zeta) \in (t_q)_\gamma$, and
  let $\pred^1_q(v)$ be the $<^1_q$-downward closure of $\bigcup_{\eta < \xi} \pred^1_{q_\eta}(v')$.

  Suppose next that $\beta \in x$ and $\beta' := \sup(x_q \cap \beta) \not\in x$.
  If $v = (\beta, \zeta) \in (t_q)_\beta$, then let $v' := (\beta', \zeta) \in (t_q)_{\beta'}$, and
  let $\pred^1_q(v) := \{v'\} \cup \pred^1_q(v')$.

  Finally, suppose that $\beta \in x$ and $\sup(x_q \cap \beta) \in x$.
  Then, for all $v \in (t_q)_\beta$, let $\pred^1_q(v)$ be the $<^1_q$-downward closure of $\bigcup_{\eta < \xi} \pred^1_{q_\eta}(v)$.

  To finish, we define $f_q$. Suppose that $\beta \in x_q \setminus (\ssup(x \cap \alpha) \cup x)$ and
  $b \in (t_q)_\beta^\lambda$. Let $\gamma_\beta := \min(x \setminus \beta)$, and let
  $b' \in (t_q)_{\gamma_\beta}^\lambda$ be given by letting $b'(i)$ be the unique $(\gamma_\beta, \zeta)$ such that $b(i) = (\beta, \zeta)$.
  Note that $b <^1_q b'$. We set
  \begin{multline*}
    \dom(f_q) := \dom(f_r) \cup \bigcup_{\eta < \xi} \dom(f_{q_\eta}) \cup \\
    \left\{(a,b) \Mid \exists \beta \in x_q \setminus (\ssup(x \cap \alpha)
    \cup x)\left(b \in (t_q)_\beta^\lambda \text{ and } (a,b') \in \bigcup_{\eta < \xi}\dom(f_{q_\eta})\right)\right\}.
  \end{multline*}

  If $(a,b) \in \dom(f_r)$, then we set $f_q(a,b) := f_r(a,b)$. If $(a,b) \in \bigcup_{\eta < \xi} \dom(f_{q_\eta}) \setminus \dom(f_r)$,
  then we let $f_q(a,b) := \bigcup_{\eta < \xi} f_{q_\eta}(a,b)$. If $(a,b)$ is such that $b \in (t_p)_\beta^\lambda$
  for some $\beta \in x_p \setminus (\ssup( x \cap \alpha) \cup x)$ and $(a,b') \in \bigcup_{\eta < \xi} \dom(f_{q_\eta})$,
  then let $f_q(a,b) = \bigcup_{\eta < \xi} f_{q_\eta}(a,b') = f_q(a,b')$. It is easily verified that $q$ is as desired.
\end{proof}

\begin{lemma}
  Suppose $p \in \bb{Q}$, $\alpha < \ssup(x_p)$, and $q \leq p \rest \alpha$ with $q \in \bb{P}_\alpha$.
  Then there is $r \in \bb{P}$ that is a greatest lower bound for $p$ and $q$. Moreover, we have
  $x_r = x_p \cup x_q$ and $r \rest \alpha = q$.
\end{lemma}

\begin{proof}
  We construct such an $r$ by doing as little as possible while still satisfying Definition \ref{p_defn}
  and extending both $p$ and $q$.
  Let $x_r := x_p \cup x_q$, and require that $r \rest \alpha = q$.
  Suppose that $\beta \in x_p \setminus \alpha$. Let $\pred^0_r(\beta) :=
  \pred^0_p(\beta) \cup \bigcup_{\gamma \in \pred^0_p(\beta) \cap \alpha} \pred^0_q(\gamma)$.
  Let $t_r := t_p \cup t_q$. If $v \in t_p \setminus t_q$, then let
  $\pred^1_r(v) := \pred^1_p(v) \cup \bigcup_{u \in \pred^1_p(v) \cap t_q} \pred^1_q(u)$.
  Finally, let $\dom(f_r) := \dom(f_p) \cup \dom(f_q)$. If $(a,b) \in \dom(f_q)$, then let
  $f_r(a,b) := f_q(a,b)$. If $(a,b) \in \dom(f_p) \setminus \dom(f_q)$, then let
  $f_r(a,b) := f_p(a,b)$.

  The only clauses of Definition \ref{p_defn} that are non-trivial to check are \eqref{p_defn6} and \eqref{p_defn7}.
  Let us first deal with Clause~\eqref{p_defn6}. To this end, fix $a,b,c \in t_r^\lambda$
  such that $(a,b), (a,c) \in \dom(f_r)$ and $f_r(a,b) \cap f_r(a,c) \neq \emptyset$.
  If we have either $(a,b), (a,c) \in \dom(f_q)$ or $(a,b), (a,c) \in \dom(f_p) \setminus \dom(f_q)$,
  then the conclusion of Clause~\eqref{p_defn6} follows from the fact that $p,q \in \bb{P}$. Thus,
  we may assume without loss of generality that $(a,b) \in \dom(f_q)$ and $(a,c) \in \dom(f_p)
  \setminus \dom(f_q)$. Let $\beta, \gamma \in x_r$ be such that $b \in (t_r)_\beta^\lambda$
  and $c \in (t_r)_\gamma^\lambda$. By assumption, we have $\beta < \alpha \leq \gamma$.

  If $\beta \not\leq^0_r \gamma$, then there is nothing to check. Thus, assume that
  $\beta \leq^0_r \gamma$. By the definition of $\leq^0_r$, it follows that there is
  $\beta' \in (x_p \cap \alpha)$ such that $\beta \leq^0_q \beta'$ and
  $\beta' \leq^0_p \gamma$. Let $c' \in (t_p)_{\beta'}^\lambda$ be the unique level sequence
  such that $a <^1_p c' <^1_p c$. Since $p \in \bb{P}$, it follows that
  $(a,c') \in \dom(f_p)$ and $f_p(a,c') \supseteq f_p(a,c)$. Since $q \leq p \rest \alpha$,
  we must have $(a,c') \in \dom(f_q)$ and $f_q(a,c') \supseteq f_p(a,c)$. Thus, we have
  $f_q(a,c') \cap f_q(a,b) \neq \emptyset$. Since $q \in \bb{P}$ and
  $\beta \leq^0_q \beta'$, we have that, for $\lambda$-many $i < \lambda$,
  $b(i) \leq^1_q c'(i)$. But then, for all such $i < \lambda$, we also have
  $b(i) \leq^1_r c(i)$, as required.

  Finally, we check Clause~\eqref{p_defn7}. Suppose that $(a,c) \in \dom(f_r)$ and $b \in t_r^\lambda$
  is such that $a <^1_r b <^1_r c$. If $(a,c) \in \dom(f_q)$, then the conclusion follows from
  the fact that $q \in \bb{P}$. Thus suppose that $(a,c) \in \dom(f_p) \setminus \dom(f_q)$.
  Let $\beta \in x_r$ be such that $b \in (t_r)_\beta^\lambda$, and let
  $\gamma \in x_p$ be such that $c \in (t_p)_\gamma^\lambda$. If $\beta \in x_p$, then
  we have $a <^1_p b <^1_p c$, and the conclusion follows from the fact that $p \in \bb{P}$.
  Thus, assume that $\beta \in x_q \setminus x_p$. Then there is $\beta' \in x_p \cap \alpha$
  such that $\beta \leq^0_q \beta'$ and $\beta' \leq^0_p \gamma$. Let $c' \in (t_p)_{\beta'}^\lambda$
  be the unique level sequence such that $a <^1_p c' <^1_p c$. Since $p \in \bb{P}$,
  we have $(a,c') \in \dom(f_p)$ and $f_p(a,c') \supseteq f_p(a,c)$. Since $q \leq_{\bb{P}} p \rest \alpha$,
  we have $f_q(a,c') \supseteq f_p(a,c)$. Finally, since $q \in \bb{P}$ and $a <^1_q b <^1_q c'$,
  we have $(a,b) \in \dom(f_q)$ and $f_q(a,b) \supseteq f_q(a,c')$. Thus, $(a,b) \in \dom(f_r)$ and
  $f_r(a,b) \supseteq f_r(a,c)$, as required.
\end{proof}

It now follows that $(\bb{P}, \leq_\bb{P}, \bb{Q})$ is in $\mathcal{P}_{\lambda^+}$.
We are thus left with isolating the relevant sharply dense systems. The following are all straightforward.

\begin{lemma}[Normal and splitting]
  Suppose $\eta < \lambda^+$.
  For every $\alpha<\beta<\lambda^{++}$, let $D^{ns}_{\eta, \{\alpha, \beta\}}$ be the set of all conditions $p \in \bb{Q}$ such that:
  \begin{itemize}
    \item $\{\alpha, \beta\} \subseteq x_p$;
    \item $(\alpha, \eta), (\beta, \eta) \in t_p$;
    \item $(\alpha, \eta)$ has at least two $<^1_p$-successors in $(t_p)_\beta$.
  \end{itemize}

   Then $\mathcal{D}^{ns}_\eta:= \{D^{ns}_{\eta, x} \mid x \in {\lambda^{++}\choose 2}\}$ is a sharply dense system.\qed
\end{lemma}

\begin{lemma}[Complete]
  Suppose that $\mu < \lambda^+$ is a regular cardinal and $g:\mu \rightarrow \lambda^+$.
  For every $x \in {\lambda^{++}\choose \mu + 1}$, let
 $D^{\text{com}}_{g, x}$ be the set of all conditions $p \in \bb{Q}$ such that:
  \begin{itemize}
    \item $x \subseteq x_p$;
    \item for all $i < \mu$, we have $(x(i), g(i)) \in t_p$;
    \item if $\{(x(i), g(i)) \mid i < \mu\}$ forms a $<^1_p$-chain,
      then it has a $<^1_p$-upper bound in $(t_p)_{x(\mu)}$.
  \end{itemize}

    Then $\mathcal{D}^{\text{com}}_g := \{D^{\text{com}}_{g, x} \mid x \in {\lambda^{++}\choose \mu + 1}\}$ is a sharply dense system.\qed
\end{lemma}

\begin{lemma}[Super-Souslin]
  Suppose $\delta, \epsilon < \lambda^+$. For all  $\alpha<\beta<\lambda^{++}$,
  let $E_{\delta, \epsilon, \{\alpha, \beta\}}$ be the set of all conditions $p \in \bb{Q}$ such that:
  \begin{itemize}
    \item $\{\alpha, \beta\} \subseteq x_p$;
    \item $a_{\alpha, \delta}, a_{\beta, \epsilon} \in t_p^\lambda$;
    \item if $a_{\alpha, \delta} <^1_p a_{\beta, \epsilon}$, then $(a_{\alpha, \delta}, a_{\beta, \epsilon}) \in \dom(f_p)$.
  \end{itemize}

  Then $\mathcal{E}_{\delta, \epsilon} := \{E_{\delta, \epsilon, x} \mid x \in {\lambda^{++}\choose 2}\}$ is a sharply dense system.\qed
\end{lemma}

By $\fa{\lambda^+}$, we can find a filter $G$ on $\bb{P}$ such that:
\begin{itemize}
  \item for every $\eta < \lambda^+$, $G$ meets $\mathcal{D}^{ns}_\eta$ everywhere;
  \item for every regular cardinal $\mu <\lambda^+$ and every function $g:\mu \rightarrow \lambda^+$,
    $G$ meets $\mathcal{D}^{\text{com}}_g$ everywhere;\footnote{Recall that by Corollary~\ref{cardinalarithmetic}, $\fa{\lambda^+}$ implies  $|{}^{<\lambda^+}\lambda^+|=\lambda^+$.}
  \item for all $\delta, \epsilon < \lambda^+$, $G$ meets $\mathcal{E}_{\delta, \epsilon}$ everywhere.
\end{itemize}
Now define a tree $(T, <_T)$ as follows. Let $T := \lambda^{++} \times \lambda^+$.
Let $(\alpha, \eta) <_T (\beta, \xi)$
iff there is $p \in G$ such that $(\alpha, \eta), (\beta, \xi) \in t_p$ and
$(\alpha, \eta) <^1_p (\beta, \xi)$.
The fact that $G$ meets $\mathcal{D}^{ns}_\eta$ everywhere for all $\eta < \lambda^+$
ensures that $(T, <_T)$ is a normal, splitting tree and $T_\alpha = \{\alpha\} \times \lambda^+$ for all $\alpha<\lambda^{++}$.
The fact that $G$ meets $\mathcal{D}^{\text{com}}_g$ everywhere for all regular $\mu \leq \lambda$ and $g:\mu \rightarrow \lambda^+$
ensures that $(T, <_T)$ is $\lambda^+$-complete.

Finally, we define a function $F:[T^\lambda]^2 \rightarrow \lambda^+$
witnessing that $(T,<_T)$ is a super-Souslin tree. Fix $\alpha < \beta < \lambda^{++}$ and
$\delta, \epsilon < \lambda^+$ such that $a_{\alpha, \delta} <_T a_{\beta, \epsilon}$. Find
$p \in G \cap E_{\delta, \epsilon, \{\alpha, \beta\}}$. Since
$p \in \bb{Q}$ and $a_{\alpha, \delta} <_T a_{\beta, \epsilon}$, it follows that
$a_{\alpha, \delta} <^1_p a_{\beta, \epsilon}$. Therefore, $(a_{\alpha, \delta}, a_{\beta, \epsilon})
\in \dom(f_p)$. Let $F(a_{\alpha, \delta}, a_{\beta, \epsilon})$ be an arbitrary element of
$f_p(a_{\alpha, \delta}, a_{\beta, \epsilon})$.

To verify that $F$ is as sought, fix $a,b,c \in T^\lambda$ such that
$(a,b), (a,c) \in [T^\lambda]^2$ and $F(a,b) = F(a,c)$. Without loss of generality, suppose
there are $\beta \leq \gamma < \lambda^{++}$ such that $b \in T_\beta^\lambda$, and
$c \in T_\gamma^\lambda$. Find $p_b \in G$ such that $(a,b) \in \dom(f_{p_b})$
and $F(a,b) \in f_{p_b}(a,b)$. Similarly, find $p_c \in G$ such that $(a,c) \in \dom(f_{p_c})$
and $F(a,c) \in f_{p_c}(a,c)$. Find $q \in G \cap \bb{Q}$ with $q \leq_{\bb{P}} p_b, p_c$.
Then $(a,b), (a,c) \in \dom(f_q)$, $F(a,b) \in f_q(a,b)$, and $F(a,c) \in f_q(a,c)$. In
particular, $f_q(a,b) \cap f_q(a,c) \neq \emptyset$. Since $q \in \bb{Q}$ it follows
that there are $\lambda$-many $i < \lambda$ such that $b(i) \leq^1_q c(i)$. But then,
for all such $i < \lambda$, we have $b(i) \leq_T c(i)$. Thus, $F$ witnesses that
$(T, <_T)$ is a $\lambda^{++}$-super-Souslin tree, so our proof of Theorem~C is now complete.

\section{Square and diamond} \label{square_sect}

In this section, we use $\square^B_\kappa$ and $\diamondsuit(\kappa)$ to construct combinatorial objects that will help us
prove Theorem~B in Section~\ref{main_thm_sect}.

\subsection{Enlarged direct limit}
In this short subsection, we introduce an ``enlarged direct limit'' operator.
This operator motivates our application of $\square^B_\kappa$ that will be carried out in the next subsection.

\begin{defn} For a linearly ordered set $(Y,\lhd)$ and a subset $Z\s Y$,
we define $\dbl_Z(Y,\lhd)$ as a linearly ordered set whose underlying set is $(Z\times\{0\})\uplus(Y\times\{1\})$,
ordered lexicographically by letting $(y,i)\lhd_l(y',i')$ iff one of the following holds:
\begin{itemize}
\item $y\lhd y'$;
\item $y=y'$ and $(i,i')=(0,1)$.
\end{itemize}
\end{defn}

The linearly ordered set $(Y,\lhd)$ we have in mind is a direct limit of a system of well-ordered sets,
and the choice of the subset $Z\s Y$ (to be doubled) will be defined momentarily.
The following is obvious.
\begin{lemma}\label{wellordered}
For all $Z\s Y$,
if $Y$ is well-ordered by $\lhd$, then,
$\dbl_Z(Y,\lhd)$  is well-ordered by $\lhd_l$.\qed
\end{lemma}

We start with a system of well-ordered sets.
Specifically, suppose that $\vec{\theta} = \langle \theta_\eta \mid \eta < \xi \rangle$ and
  $\vec{\pi} = \langle \pi_{\eta, \eta'} \mid \eta < \eta' < \xi \rangle$ are such that:
  \begin{itemize}
    \item   $\xi$ is a limit ordinal;
    \item $\vec\theta$ is a non-decreasing sequence of ordinals;
    \item for all $\eta < \eta' < \xi$, $\pi_{\eta, \eta'}:\theta_\eta \rightarrow \theta_{\eta'}$
      is a \coi;
    \item for all $\eta < \eta' < \eta'' < \xi$, we have $\pi_{\eta, \eta''} = \pi_{\eta', \eta''} \circ
      \pi_{\eta, \eta'}$.
    \end{itemize}

As is well-known, the direct limit of the  system $(\vec{\theta}, \vec{\pi})$ is defined as follows:
\begin{itemize}
\item Put $X:=\{(\eta,\gamma)\mid \eta<\xi, \gamma<\theta_\eta\}$.
\item For $(\eta,\gamma),(\eta',\gamma')\in X $ with $\eta<\eta'$,
let  $(\eta,\gamma)\sim(\eta',\gamma')$ iff $\pi_{\eta,\eta'}(\gamma)=\gamma'$.
\item Let  $Y$ consists of all equivalence classes $[(\eta,\gamma)]$ for $(\eta,\gamma)\in X$.
\item Order $Y$ by letting $[(\eta_0,\gamma_0)]\lhd [(\eta_1,\gamma_1)]$ iff there exists some
  $\eta\ge\max\{\eta_0,\eta_1\}$ and $\gamma_0'<\gamma_1'$ such that $(\eta_0,\gamma_0)\sim(\eta,\gamma_0')$
  and $(\eta_1,\gamma_1)\sim(\eta,\gamma_1')$.
\item  For each $\eta < \xi$, define a  map
  $\pi_\eta:\theta_\eta \rightarrow Y$ by stipulating $\pi_\eta(\gamma):=[(\eta,\gamma)]$.
\end{itemize}

\begin{defn}[Direct limit]\label{defdirectlimit}
$\lim(\vec{\theta}, \vec{\pi})$ stands for $(Y,\lhd,\langle \pi_\eta\mid \eta<\xi\rangle)$.
\end{defn}

  Next, we let $Z$ be the set of equivalence classes in $Y$ such that, for every representative
  $(\eta, \gamma)$ from the equivalence class, $\pi_\eta \restriction \gamma$ is bounded below
  $\pi_\eta(\gamma)$, i.e.,
$$Z:=\{z\in Y\mid\forall (\eta, \gamma)\in z\exists y \in Y\forall\beta<\gamma[
\pi_\eta(\beta) \lhd y \lhd \pi_\eta(\gamma)]\}.$$

Let $W := \dbl_Z(Y,\lhd)$,
and let $\varpi$ denote the  map from
 $Y$ to its canonical copy inside
 $W$, i.e., $\varpi(y)=(y,1)$.

\begin{defn}[Enlarged direct limit]
$\lim^*(\vec{\theta}, \vec{\pi})$ stands for $(W,\lhd_l,\langle \pi_\eta^*\mid \eta<\xi\rangle)$,
where   $\pi^*_\eta:=\varpi\circ\pi_\eta$ for each $\eta<\xi$.
\end{defn}

Finally, by Lemma~\ref{wellordered},
in the special case that $(Y,\lhd)$ is well-ordered,
we know that $(W,\lhd_l)$ is well-ordered.
In this case, we put $\theta:=\otp(W,\lhd_l)$, and let $\pi^*:W\rightarrow\theta$ be the collapse map.
Then, we define:
\begin{defn}[Ordinal enlarged direct limit]
$\lim^+(\vec\theta,\vec\pi)$ stands for $(\theta,\in,\langle\pi_\eta^+\mid \eta<\xi\rangle)$,
where  $\pi^+_\eta:=\pi^*\circ\pi_\eta^*$ for all $\eta<\xi$.
\end{defn}

\subsection{Square} \label{square_subsection}
Fix a $\square^B_\kappa$-sequence, $\langle C_\beta \mid \beta \in \Gamma \rangle$.
Enlarge the preceding to a sequence $\vec{C} = \langle C_\beta \mid \beta <\kappa \rangle$
by letting, for all limit $\beta\in\kappa\setminus\Gamma$, $C_\beta$ be an arbitrary club in $\beta$ of order type $\cf(\beta)$,
and letting $C_{\beta+1}:=\{0,\beta\}$ for all $\beta<\kappa$. In particular, for every $\beta\in E^\kappa_\omega\setminus\Gamma$,
we have $\acc(C_\beta)=\emptyset$. Thus, without loss of generality, we may assume that $E^\kappa_\omega\s\Gamma$.
For convenience, assume also that $0 \in C_\beta$ for all nonzero $\beta<\kappa$.

\medskip

We now turn to constructing a matrix $\vec{B} = \langle B^\beta_\eta \mid \beta < \kappa^+, ~ \eta < \kappa \rangle$
such that $\bigcup_{\eta<\kappa}B^\beta_\eta=\beta+1$ for all $\beta<\kappa^+$.
From this matrix, for each $\beta<\kappa^+$, we shall derive the following additional objects:
\begin{itemize}
\item[$\circ$] we shall let $\eta_\beta$ denote the least $\eta<\kappa$ such that $B^\beta_\eta\neq\emptyset$;
\item[$\circ$] for each $\xi\in\acc(\kappa\setminus\eta_\beta)$, we write $B^\beta_{<\xi}:=\bigcup_{\eta < \xi} B^\beta_\eta$;
\item[$\circ$] for each $\eta < \kappa$, we shall set $\theta^\beta_\eta := \otp(B^\beta_\eta)$ and let
$\pi^\beta_\eta : B^\beta_\eta \rightarrow \theta^\beta_\eta$ denote the unique order-preserving bijection;
\item[$\circ$] for each $\eta < \xi < \kappa$, $\pi^\beta_{\eta, \xi}: \theta^\beta_\eta \rightarrow \theta^\beta_\xi$ will denote the order-preserving injection indicating how $B^\beta_\eta$
``sits inside'' $B^\beta_\xi$, i.e., $\pi^\beta_{\eta, \xi} := \pi^\beta_\xi \circ (\pi^\beta_\eta)^{-1}$.
\end{itemize}

We shall also derive a ``distance function'' $d:[\kappa^+]^2\rightarrow\kappa$ by letting for all $\alpha<\beta<\kappa^+$:
$$d(\alpha,\beta):=\min\{\eta<\kappa\mid \alpha\in B^\beta_\eta\}.$$

\begin{lemma} \label{matrix_lemma}
There exists a matrix $\vec{B} = \langle B^\beta_\eta \mid
\beta < \kappa^+, ~ \eta < \kappa \rangle$ such that, for each $\beta < \kappa^+$, the following hold:
  \begin{enumerate}
    \item\label{cb1} $\langle B^\beta_\eta \mid \eta < \kappa \rangle$ is a $\subseteq$-increasing sequence of closed sets, each of size $<\kappa$,
      that converges to $\beta + 1$, and $\beta\in B^\beta_{\eta_\beta}$;
    \item\label{cb3} for all $\eta<\kappa$ and  $\alpha \in B^\beta_\eta$, we have $B^\alpha_\eta=B^\beta_\eta \cap (\alpha + 1)$ and $\pi^\alpha_\eta=\pi^\beta_\eta\restriction(\alpha+1)$;
    \item\label{cbinfty} for all $\eta<\kappa$, if $\cf(\beta)=\kappa$, then $\max(B^\beta_\eta\cap\beta)=C_\beta(\omega \eta)$;
    \item\label{cb25} if $\beta\in\Gamma\cap E^{\kappa^+}_{<\kappa}$, then $\eta_\beta=\otp(\acc(C_\beta))$ and $\acc(C_\beta)\s B^\beta_{\eta_\beta}$;
    \item\label{cb4} for all $\xi\in\acc(\kappa\setminus\eta_\beta)$, all of the following hold:
    \begin{enumerate}
    \item\label{cb5} $B^\beta_\xi$ is the ordinal closure of $B^\beta_{<\xi}$;
    \item\label{cb6} for every $\alpha \in B^\beta_\xi \setminus B^\beta_{<\xi}$,
      letting $\gamma := \min(B^\beta_\xi \setminus (\alpha + 1))$, we have $\cf(\gamma) = \kappa$ and $\alpha = C_\gamma(\omega \xi)$;
      \item\label{cofinality_lemma} $\cf(\beta) = \kappa$ iff $\ssup(\pi^\beta_{\eta, \xi}``\pi^\beta_\eta(\beta)) < \pi^\beta_\xi(\beta)$ for all $\eta \in [\eta_\beta, \xi)$.
  \end{enumerate}
\end{enumerate}
\end{lemma}
\begin{proof} The construction is by recursion on $\beta<\kappa^+$.

\begin{description}
  \item[Case 0] \textbf{$\beta = 0$.} Set $B^\beta_\eta := \{0\}$ for all $\eta < \kappa$. It is trivial to see that Clauses \eqref{cb1}--\eqref{cb4} all hold.
  \item[Case 1] \textbf{$\beta = \alpha + 1$.} For all $\eta<\eta_\alpha$, let $B^\beta_\eta:=\emptyset$,
  and for all $\eta\in[\eta_\alpha,\kappa)$, let $B^\beta_\eta := \{\beta\}\cup B^\alpha_\eta$.
  It is trivial to see that Clauses \eqref{cb1}--\eqref{cb4} all hold.
  \item[Case 2] \textbf{$\beta \in\acc(\kappa)$ and $\sup(\acc(C_\beta))<\beta$.}
    In particular, $a:=C_\beta\setminus\sup(\acc(C_\beta))$ is a cofinal subset of $\beta$ of order type $\omega$.
    Note that, since $\cf(\beta)=\omega$, we have $\beta\in\Gamma$.
    Put $\eta_\beta:=\otp(\acc(C_\beta))$ and $\eta^*:=\max\{\eta_\beta,\sup(d``[a]^2)\}$.
    Now, for all $\eta<\kappa$, define $B^\beta_\eta$ as follows:
      \begin{itemize}
      \item[$\br$] If $\eta<\eta_\beta$, then let $B^\beta_\eta:=\emptyset$. Clauses \eqref{cb3}--\eqref{cb4} are trivially satisfied.
        \item[$\br$] If $\eta_\beta\le\eta\le\eta^*$, then let $\alpha^*:=\min(a)$ and put $B^\beta_\eta := \{\beta\} \cup B^{\alpha^*}_\eta$.
        Since $\alpha^*\in\acc(C_\beta)\cup\{0\}$ and $\beta\in\Gamma$, we have $C_{\alpha^*}\sq C_\beta$,
        which ensures Clause~\eqref{cb25}. As for Clause~\eqref{cofinality_lemma}, for all $\eta < \xi$ in $[\eta_\beta, \eta^*]$, we have
        $$\ssup(\pi^\beta_{\eta, \xi}``\pi^\beta_\eta(\beta)) = \pi^\beta_{\eta, \xi}(\pi^\beta_\eta(\alpha^*) + 1) = \pi^\beta_\xi(\alpha^*) + 1 = \pi^\beta_\xi(\beta).$$
       The other clauses are easily seen to be satisfied.
        \item[$\br$] Otherwise, let $B^\beta_\eta := \{\beta\} \cup \bigcup_{\alpha\in a} B^{\alpha}_\eta$.
        Since $\eta>\eta^*$, for every pair of ordinals $\alpha<\alpha'$ from $a$, we have $\alpha\in B^{\alpha'}_\eta$, so that $B^{\alpha}_\eta=B^{\alpha'}_\eta\cap(\alpha+1)$.
        It follows that $\langle B_\eta^{\alpha}\mid \alpha\in a\rangle$ is an $\sq$-increasing sequence of closed sets.
        In particular,  $B_\eta^\beta\cap\beta$ is a club in $\beta$, which takes care of Clause~\eqref{cofinality_lemma}. So, all clauses are satisfied.
      \end{itemize}
  \item[Case 3] \textbf{$\cf(\beta)<\kappa$ and $\sup(\acc(C_\beta))=\beta$.}
        Put $\eta_\beta:=\sup(d``[\acc(C_\beta)]^2)$, and, for all $\eta<\kappa$, define $B^\beta_\eta$ as follows:
      \begin{itemize}
      \item[$\br$] If $\eta<\eta_\beta$, then let $B^\beta_\eta:=\emptyset$. Clauses \eqref{cb3}--\eqref{cb4} are trivially satisfied.
      \item[$\br$] If $\eta\ge\eta_\beta$, then let $B^\beta_\eta := \{\beta\} \cup \bigcup_{\alpha \in \acc(C_\beta)} B^\alpha_\eta$.
Since $\eta\ge\sup(d``[\acc(C_\beta)]^2)$, we have that $\langle B_\eta^{\alpha}\mid \alpha\in\acc(C_\beta)\rangle$ is an $\sq$-increasing sequence of closed sets.
So $B_\eta^\beta\cap\beta$ is a club in $\beta$, and all clauses except Clause~\eqref{cb25} are easily seen to be satisfied.
Now, if $\beta\in\Gamma$, then, since Clause~\eqref{cb25} holds for all $\alpha\in\acc(C_\beta)$, we have $\eta_\beta=\otp(\acc(C_\beta))$, so that Clause~\eqref{cb25} holds for $\beta$, as well.
      \end{itemize}
  \item[Case 4] \textbf{$\beta \in E^{\kappa^+}_\kappa$.}
    For all $\eta < \kappa$, let $\alpha_\eta := C_\beta(\omega\eta)$  and $B^\beta_\eta := \{\beta\} \cup B^{\alpha_\eta}_\eta$, so that Clause~\eqref{cbinfty} is satisfied.

    Since $\cf(\beta)=\kappa$, we have $\beta\in\Gamma$. Hence, for all $\eta<\xi<\kappa$, we have $\alpha_\xi\in\Gamma$,
    so that $\alpha_\eta\in B^{\alpha_\xi}_\xi$ by Clause~\eqref{cb25}.
    It follows that $\langle B^\beta_\eta\cap\beta\mid \eta<\kappa\rangle$ is $\subseteq$-increasing.
    In particular, Clauses \eqref{cb1} and \eqref{cb3} are satisfied.
    It also follows that, for all $\xi\in\acc(\kappa)$
    and $\alpha\in B^\beta_\xi\cap\beta$, we have $B^\beta_{<\xi}\cap(\alpha+1)=B^\alpha_{<\xi}$,
    so that Clauses \eqref{cb5} and \eqref{cb6} are satisfied.

    Finally, to verify Clause~\eqref{cofinality_lemma}, fix an arbitrary $\xi\in\acc(\kappa)$ and $\eta<\xi$.
  By Clauses~\eqref{cb3} and \eqref{cbinfty}, we have $B^\beta_\eta = \{\beta\} \cup B^{\alpha_\eta}_\eta$ and  $B^\beta_\xi = \{\beta\} \cup B^{\alpha_\xi}_\xi$,
  so that  $\pi^\beta_\eta(\beta) = \pi^\beta_\eta(\alpha_\eta) + 1$ and $\pi^\beta_\xi(\beta) = \pi^\beta_\xi(\alpha_\xi) + 1$.
  Therefore, we have
  \[
    \pi^\beta_{\eta, \xi}``\pi^\beta_\eta(\beta) \subseteq \pi^\beta_{\eta, \xi}(\pi^\beta_\eta(\alpha_\eta)) + 1
    = \pi^\beta_\xi(\alpha_\eta) + 1 < \pi^\beta_\xi(\alpha_\xi) < \pi^\beta_\xi(\beta).
    \qedhere
  \]

\end{description}
\end{proof}

The next lemma assumes familiarity with the previous subsection.

\begin{lemma} \label{direct_limit_lemma}
  Suppose $\beta < \kappa^+$ and $\xi \in \acc(\kappa \setminus \eta_\beta)$. Write $\vec{\theta} :=
  \langle \theta^\beta_\eta \mid \eta < \xi \rangle$ and $\vec{\pi} := \langle \pi^\beta_{\eta, \eta'}
  \mid \eta < \eta' < \xi \rangle$. Then $\lim^+(\vec{\theta}, \vec{\beta})$ is defined and,
  letting $(\theta, \in, \langle \pi^+_\eta \mid \eta < \xi \rangle):=\lim^+(\vec\theta,\vec\pi)$, we have
  $\theta = \theta^\beta_\xi$ and $\langle \pi^+_\eta \mid \eta < \xi \rangle = \langle \pi^\beta_{\eta, \xi}
  \mid \eta < \xi \rangle$.
\end{lemma}

\begin{proof}
  Let $\langle Y,\lhd,\langle \pi_\eta\mid\eta<\xi\rangle) := \lim((\vec{\theta}, \vec{\pi}))$.
  For every class $y\in Y$  and representatives $(\eta,\gamma),(\eta',\gamma')\in y$
  with $\eta<\eta'$, we have $\pi_{\eta,\eta'}^\beta(\gamma)=\gamma'$,
  i.e., $(\pi^\beta_{\eta'})^{-1}(\gamma')=(\pi^\beta_\eta)^{-1}(\gamma)$.
  Therefore, for each $y\in Y$, we may let $\alpha_y:=(\pi_\eta^\beta)^{-1}(\gamma)$ for an arbitrary choice of $(\eta,\gamma)\in y$.
  Note that, for all  $y,y'$ in $Y$, we have
  $y\lhd y'$ iff $\alpha_y < \alpha_{y'}$. Therefore,
  the order type of $(Y,\lhd)$ is precisely $\otp(B^\beta_{<\xi})$.
  In particular, $\lim(\vec{\theta}, \vec{\pi})$ is well-ordered, so $\lim^+(\vec{\theta}, \vec{\pi})$
  is defined. Write $(\theta,\in, \langle \pi^+_\eta \mid \eta < \xi \rangle)$ for $\lim^+(\vec{\theta}, \vec{\pi})$.

  Let $Z$ be the set of equivalence classes in $Y$ such that,
  for every representative $(\eta, \gamma)$ from the class, we have that $\pi_\eta \restriction \gamma$
  is bounded below $\pi_\eta(\gamma)$. By Clause~\eqref{cofinality_lemma} of Lemma~\ref{matrix_lemma},
  we know that $Z=\{  z \in Y\mid \cf(\alpha_z) = \kappa\}$.

  For all $\alpha \in B^\beta_{<\xi} \cap E^{\kappa^+}_\kappa$, we have $C_\alpha(\omega\xi) \in B^\beta_\xi \setminus B^\beta_{<\xi}$
  and, moreover, $\alpha = \min(B^\beta_\xi \setminus (C_\alpha(\omega\xi) + 1))$.
  Also, by Clauses \eqref{cb5} and \eqref{cb6} of Lemma~\ref{matrix_lemma}, we know that $B^\beta_\xi =
  B^\beta_{<\xi} \cup \{C_\alpha(\omega\xi) \mid \alpha \in B^\beta_{<\xi} \cap E^{\kappa^+}_\kappa\}$.
  Now, for all $z \in Z$, the addition of $(z,0)$ when passing from $Y$
  to $W:=\dbl_Z(Y,\lhd)$ corresponds precisely to the addition of $C_{\alpha_z}(\omega\xi)$
  when passing from $B^\beta_{<\xi}$ to $B^\beta_\xi$. It follows that $\otp(W,\lhd_l)
  = \otp(B^\beta_\xi) = \theta^\beta_\xi$. That is, $\theta = \theta^\beta_\xi$. Letting $\pi^*:W\rightarrow\theta$ be the collapse map, we have that, for all $z \in Y$,
  $\pi^*(z,1) = \pi^\beta_\xi(\alpha_z)$, so, for all $\eta < \xi$ and $\gamma < \theta^\beta_\eta$,
  we have
  \[
    \pi^+_\eta(\gamma) = \pi^*([(\eta, \gamma)],1) = \pi^\beta_\xi((\pi^\beta_\eta)^{-1}(\gamma)) = \pi^\beta_{\eta, \xi}(\gamma),
  \]
  so $\langle \pi^+_\eta \mid \eta < \xi \rangle = \langle \pi^\beta_{\eta, \xi} \mid \eta < \xi \rangle$.
\end{proof}

\subsection{Diamond}
Our next goal is to prove the following.
\begin{lemma}\label{lemma21}
  Suppose that $\diamondsuit(\kappa)$ holds
  and that $(\bb{P}, \leq_{\bb{P}}, \bb{Q}) \in \mathcal{P}_\kappa$.
  Then there are arrays $\langle \vartheta^\xi_\eta \mid \eta \leq \xi<\kappa\rangle$,
  $\langle \varpi^\xi_{\eta, \eta'} \mid \eta < \eta' \leq \xi < \kappa \rangle$,
  and $\langle q^\xi_\eta \mid \eta < \xi < \kappa \rangle$ satisfying the following:
  For every $\beta<\kappa^+$ and every decreasing sequence
  $\langle p_\eta \mid \eta < \kappa \rangle\in\prod_{\eta<\kappa}\bb{P}_{B^\beta_\eta}$,
  there are stationarily many $\xi<\kappa$ such that:
  \begin{itemize}
    \item $\langle \vartheta^\xi_\eta \mid \eta \leq \xi \rangle = \langle \theta^\beta_\eta \mid \eta \leq \xi \rangle$;
    \item $\langle \varpi^\xi_{\eta, \eta'} \mid \eta < \eta' \leq \xi \rangle = \langle \pi^\beta_{\eta, \eta'} \mid \eta < \eta' \leq \xi \rangle$;
    \item $\langle q^\xi_\eta \mid \eta < \xi \rangle = \langle \pi^\beta_\eta.p_\eta \mid \eta < \xi \rangle$.
  \end{itemize}
\end{lemma}

The rest of this subsection will be devoted to proving Lemma~\ref{lemma21}.
To avoid the use of codings, we shall make use of the following equivalent version of $\diamondsuit(\kappa)$ (see~\cite{paper22}).

\begin{defn}\label{def_Diamond_H_kappa_minus} $\diamondsuit^-(H_\kappa)$ asserts the existence of a sequence
$\langle A_\xi \mid \xi < \kappa \rangle$ such that, for every $A\subseteq H_\kappa$ and $p\in H_{\kappa^{+}}$,
there exists an elementary submodel $\mathcal M\prec H_{\kappa^{+}}$, with $p\in\mathcal M$, such
that $\kappa^{\mathcal M}:=\mathcal M\cap\kappa$ is an ordinal $<\kappa$ and $A\cap \mathcal M=A_{\kappa^{\mathcal M}}$.
\end{defn}

Fix a $\diamondsuit^-(H_\kappa)$-sequence, $\langle A_\xi\mid\xi<\kappa \rangle$.
\begin{defn} \label{good_defn}
  We say that $\xi<\kappa$ is \emph{good} if  $\xi \in \acc(\kappa)$ and
  $$A_\xi = \{(\vartheta^\xi_\eta,q^\xi_\eta,\varpi^\xi_{\eta, \eta'},\eta,\eta') \mid \eta < \eta' < \xi\},$$
  where, for all $\eta < \eta' < \eta''< \xi$, we have
  \begin{itemize}
    \item $\vartheta^\xi_\eta \leq \vartheta^\xi_{\eta'} < \kappa$;
    \item $q^\xi_\eta \in \bb{P}_{\vartheta^\xi_\eta}$;
    \item $\varpi^\xi_{\eta, \eta'}: \vartheta^\xi_\eta \rightarrow \vartheta^\xi_{\eta'}$ is a \coi,
      and $q^\xi_{\eta'} \leq_{\bb{P}} \varpi^\xi_{\eta, \eta'}.q^\xi_\eta$;
    \item $\varpi^\xi_{\eta, \eta''} = \varpi^\xi_{\eta', \eta''} \circ \varpi^\xi_{\eta, \eta'}$;
    \item $\lim(\langle \vartheta^\xi_\eta \mid \eta < \xi \rangle, \langle \varpi^\xi_{\eta, \eta'} \mid \eta < \eta' < \xi \rangle)$
      is well-ordered.
  \end{itemize}
\end{defn}

$\br$ If $\xi<\kappa$ is good, then
$\langle \vartheta^\xi_\eta \mid \eta <\xi\rangle$,
  $\langle \varpi^\xi_{\eta, \eta'} \mid \eta < \eta' <\xi\rangle$,
  and $\langle q^\xi_\eta \mid \eta < \xi \rangle$ are already defined, and we let:
\[
(\vartheta^\xi_\xi, \in, \langle \varpi^\xi_{\eta, \xi} \mid \eta < \xi \rangle)
:=\lim {^+}(\langle \vartheta^\xi_\eta \mid \eta < \xi \rangle, \langle \varpi^\xi_{\eta, \eta'}
\mid \eta < \eta' < \xi \rangle)
.\]

$\br$ If $\xi < \kappa$ is not good, then let $\langle \vartheta^\xi_\eta \mid \eta \leq \xi \rangle$, $\langle \varpi^\xi_{\eta, \eta'}
\mid \eta < \eta' \leq \xi \rangle$, and $\langle q^\xi_\eta \mid \eta < \xi \rangle$ be arbitrary.

We claim that the arrays thus defined satisfy the conclusion of Lemma~\ref{lemma21}. To verify this,
fix $\beta < \kappa^+$, a decreasing sequence $\langle p_\eta \mid \eta < \kappa \rangle\in\prod_{\eta<\kappa}\bb{P}_{B^\beta_\eta}$,
and a club $D$ in $\kappa$. Put
$$A:=\{(\theta^\beta_\eta,\pi^\beta_\eta.p_\eta,\pi^\beta_{\eta, \eta'},\eta,\eta') \mid \eta <   \eta' < \kappa \}.$$
Since $A\s H_\kappa$ and $D\in H_{\kappa^+}$, we can let $p:=\{A,D\}$ and fix an elementary submodel $\mathcal M\prec H_{\kappa^+}$ with $p\in\mathcal M$
such that $\xi:=\mathcal M\cap\kappa$ is in $\kappa$ and $A\cap\mathcal M=A_\xi$.
By the fact that $D\in\mathcal M$ and the elementarity of $\mathcal M$, we have $\xi\in D$.
Since $\mathcal M\cap\kappa=\xi$ and $A\in\mathcal M$, and by the elementarity of $\mathcal M$, we have
$$A_\xi=\{(\theta^\beta_\eta,\pi^\beta_\eta.p_\eta,\pi^\beta_{\eta, \eta'},\eta,\eta') \mid \eta <   \eta' < \xi \}.$$
In particular, $\xi$ is good. By Lemma~\ref{direct_limit_lemma}, we have $\vartheta^\xi_\xi = \theta^\beta_\xi$ and,
for all $\eta < \xi$, $\varpi^\xi_{\eta, \xi} = \pi^\beta_{\eta, \xi}$. Therefore, $\xi \in D$
satisfies the three bullet points in the statement of Lemma~\ref{lemma21}. Since $D$ was arbitrary,
this completes the proof of the lemma.

\section{Proof of Theorem~B} \label{main_thm_sect}

This section is devoted to the proof of Theorem~B, which forms the main result of this paper.

\begin{thmb} \label{main_thm}
  Suppose that $\square^B_\kappa$ and $\diamondsuit(\kappa)$ both hold.
  Then so does $\fa{\kappa}$.
\end{thmb}

\subsection{Setup}
Fix an arbitrary $(\bb{P}, \leq_{\bb{P}}, \bb{Q}) \in \mathcal{P}_\kappa$
along with a collection $\{\mathcal{D}_i \mid i < \kappa\}$ of sharply dense systems.
For each $i < \kappa$, write $\mathcal{D}_i = \{D_{i, x} \mid x \in {\kappa^+ \choose \theta_{\mathcal{D}_i}}\}$.

Let $\vec B$ be given by Lemma~\ref{matrix_lemma},
and let $\langle \vartheta^\xi_\eta \mid \eta \leq \xi < \kappa \rangle$,
$\langle \varpi^\xi_{\eta, \eta'} \mid \eta < \eta' \leq \xi < \kappa \rangle$,
and $\langle q^\xi_\eta \mid \eta < \xi < \kappa \rangle$
be given by Lemma \ref{lemma21} applied to $(\bb{P}, \leq_{\bb{P}}, \bb{Q})$.
\begin{defn}Let $X$ denote the set of $\xi \in \acc(\kappa)$ such that:
\begin{itemize}
  \item $\xi$ is good, in the sense of Definition~\ref{good_defn};
  \item $\langle \varpi^\xi_{\eta, \xi}.q^\xi_\eta \mid \eta < \xi \rangle$
    admits a lower bound in $\bb{P}_{\vartheta^\xi_\xi}$.
\end{itemize}
\end{defn}

Let $\lhd_\kappa$ be some well-ordering of $H_\kappa$.
Using $\kappa^{<\kappa}=\kappa$ (which follows from $\diamondsuit(\kappa)$), enumerate all elements of
$\bigcup_{i<\kappa}\{i\}\times\kappa\times {\kappa\choose \theta_{\mathcal{D}_i}}$ as a sequence $\langle (i_\eta, j_\eta, z_\eta) \mid \eta < \kappa \rangle$.

\begin{lemma}\label{s_xi_proposition}
  There is a sequence of conditions $\langle s_\xi\mid \xi \in X \rangle\in\prod_{\xi \in X}\bb{Q}_{\vartheta^\xi_\xi}$,
  such that, for all $\xi\in X$:
  \begin{itemize}
    \item $s_\xi$ is a lower bound for $\langle \varpi^\xi_{\eta, \xi}.q^\xi_\eta \mid \eta < \xi \rangle$;
    \item for all $\eta < \xi$, if $j_\eta < \xi$ and $z_\eta \subseteq \vartheta^\xi_{j_\eta}$,
    then there is $q \in D_{i_\eta, \varpi^\xi_{j_\eta, \xi}``z_\eta}$ such that $s_\xi \leq_{\bb{P}} q$.
  \end{itemize}
\end{lemma}
\begin{proof} Let $\xi \in X$ be arbitrary.
We first define a sequence $\langle s^\eta \mid \eta \le \xi \rangle\in\prod_{\eta\le \xi}\bb{Q}_{\vartheta_\xi^\xi}$ by recursion on $\eta$:

$\br$ For $\eta=0$, use Clauses \eqref{c1.5} and \eqref{c9} of Definition~\ref{class_def} and
the fact that $\xi \in X$ to find
$s^0 \in \bb{Q}$ such that $x_{s^0} = \vartheta^\xi_\xi$ and $s^0$ is a lower bound
for $\langle \varpi^\xi_{\eta, \xi}.q^\xi_\eta \mid \eta < \xi \rangle$.

$\br$ For $\eta < \xi$, with $j_\eta < \xi$ and $z_\eta \subseteq \vartheta^\xi_{j_\eta}$,
use the fact that $\mathcal{D}_{i_\eta}$ is a sharply dense system and that $\varpi^\xi_{j_\eta, \xi}``z_\eta
\subseteq \vartheta^\xi_\xi$ to find $s^{\eta, *} \in D_{i_\eta, \varpi^\xi_{j_\eta, \xi}``z_\eta}$
such that $s^{\eta, *} \leq_\bb{P} s^\eta$ and $x_{s^{\eta, *}} = \vartheta^\xi_\xi$.
Then, use Clause~\eqref{c9} of Definition~\ref{class_def} to find $s^{\eta + 1} \in \bb{Q}$ such that
$s^{\eta + 1} \leq_{\bb{P}} s^{\eta,*}$ and $x_{s^{\eta + 1}} = \vartheta_\xi^\xi$.

$\br$ For $\eta<\xi$ with $j_\eta \geq \xi$ or $z_\eta \not\subseteq \vartheta^\xi_{j_\eta}$, simply let
$s^{\eta + 1} := s^\eta$.

$\br$ For $\eta\in\acc(\xi+1)$, assuming that $\langle s^{\zeta}\mid \zeta<\eta\rangle$ has already been defined,
use Clause~\eqref{c10} of Definition~\ref{class_def} to let $s^\eta$ be a lower bound for $\langle s^\zeta \mid \zeta < \eta \rangle$
in $\bb{Q}$ with $x_{s^\eta} = \vartheta_\xi^\xi$.

Having constructed $\langle s^\eta\mid\eta\le \xi\rangle$, it is clear that $s_\xi:=s^\xi$ is as sought.
\end{proof}

Fix a sequence $\langle s_\xi \mid \xi \in X\rangle$ as in the preceding lemma.
We will construct a matrix of conditions $\langle p^\beta_\eta \mid \beta < \kappa^+, ~ \eta < \kappa \rangle$ satisfying:
\begin{enumerate}[(i)]
  \item\label{cc1} for all $\beta < \kappa^+$ and $\eta < \kappa$, we have $p^\beta_\eta \in \bb{P}_{B^\beta_\eta}$;
  \item\label{cc2} for all $\beta < \kappa^+$, $\langle p^\beta_\eta \mid \eta<\kappa\rangle$ is $\leq_{\bb{P}}$-decreasing;
  \item\label{cc3} for all $\beta < \kappa^+$, $\eta < \kappa$, and $\alpha \in B^\beta_\eta$, we have $p^\beta_\eta \rest (\alpha + 1) = p^\alpha_\eta$;
  \item\label{cc4} for all $\beta < \kappa^+$, all $i < \kappa$, and all  $x \in {\beta + 1 \choose \theta_{\mathcal{D}_i}}$,
    there is $\xi < \kappa$ and $q \in D_{i, x}$ such that $p^\beta_\xi \leq_\bb{P} q$;
  \item\label{cc5} for all $\beta \in E^{\kappa^+}_\kappa$ and all $\xi\in\acc(\kappa)$, the sequence $\langle \pi^\beta_\xi.p^\beta_\eta \mid \eta \leq \xi \rangle$
depends only on the value of $C_\beta(\omega\xi)$.
\end{enumerate}

Note that if we are successful, then, letting $G$ be the upward closure of $\{ p^\beta_\eta \mid \beta < \kappa^+,\allowbreak ~ \eta < \kappa \}$,
it follows from Clauses \eqref{cc1}--\eqref{cc4} that $G$ is a filter on $\bb{P}$ that, for each $\eta < \kappa$, meets $\mathcal{D}_\eta$ everywhere.
Of course, the sequence $\langle s_\xi\mid \xi\in X\rangle$, which was derived from $\diamondsuit$, will be a key to ensuring Clause~\eqref{cc4}.

\subsection{Hypotheses}

The construction of $\langle p^\beta_\eta \mid \beta < \kappa^+, ~ \eta < \kappa \rangle$
will be by recursion on $\eta<\kappa$ and, for fixed $\eta$, by recursion on $\beta<\kappa^+$.
We will maintain requirements \eqref{cc1}--\eqref{cc3} and \eqref{cc5} as recursion hypotheses. In order to ensure that the
construction will be successful, we need to carry along some further hypotheses.
Suppose that $\beta < \kappa^+$, $\xi \in \acc(\kappa)$,
and $\langle p^\alpha_\eta \mid \alpha < \kappa^+, ~ \eta < \xi \rangle$ has been constructed.

\begin{defn}\label{active_defn}
  We say that the pair $(\beta, \xi)$ is \emph{active} if $\xi \in X$, $\theta^\beta_\xi \leq \vartheta_\xi^\xi$,
  and one of the following holds:
  \begin{itemize}
      \item $\xi > \eta_\beta$ and, for all $\eta < \xi$, $s_\xi \leq_{\bb{P}} \pi^\beta_\xi.p^\beta_\eta$; or,
      \item $\xi = \eta_\beta$ and there is $\gamma \in E^{\kappa^+}_\kappa$ such that
        $\beta \in \acc(C_\gamma)$ and $(\gamma, \xi)$ is active.
  \end{itemize}
\end{defn}

In our construction, we will require that, for all active $(\beta, \xi)$, we have $\pi^\beta_\xi.p^\beta_\xi
= s_\xi \rest \theta^\beta_\xi$. In particular, if $(\beta, \xi)$ is active, then
$p^\beta_\xi \in \bb{Q}$ and $x_{p^\beta_\xi} = B^\beta_\xi$. Moreover, for all $\beta < \kappa^+$, we will arrange that,
if $\xi < \kappa$ is least such that $\beta \in x_{p^\beta_\xi}$, then either $(\beta, \xi)$ is active
or ($\xi = \eta_\beta$ and $p^\beta_\xi \in \bb{Q}$).

\begin{lemma} \label{active_lemma}
  Suppose that $\beta < \kappa^+$, $\xi \in X$, and $(\beta, \xi)$ is active.
  Then $(\alpha, \xi)$ is active for all $\alpha \in B^\beta_\xi$.
\end{lemma}

\begin{proof} Let $\alpha \in B^\beta_\xi$ be arbitrary.
  As $B^\beta_\xi \cap (\alpha + 1) = B^\alpha_\xi$, we have $\theta^\alpha_\xi < \theta^\beta_\xi \leq \vartheta_\xi^\xi=x_{s_\xi}$.

  $\br$ If $\xi > \eta_\beta$ and  $\alpha \in B^\beta_{<\xi}$, then  $\xi > \eta_\alpha$ and,
  for all sufficiently large $\eta < \xi$, we have $p^\alpha_\eta = p^\beta_\eta \rest (\alpha + 1)$.
  By Clause~\eqref{cb3} of Lemma~\ref{matrix_lemma}, then,
  $s_\xi \leq_{\bb{P}} \pi^\beta_\eta.p^\alpha_\eta = \pi^\alpha_\eta.p^\alpha_\eta$, so $(\alpha, \xi)$ is active.

  $\br$ If $\xi > \eta_\beta$ and $\alpha \in B^\beta_\xi \setminus B^\beta_{<\xi}$, then let
  $\gamma := \min(B^\beta_\xi \setminus (\alpha + 1))$. By Clause~\eqref{cb6} of Lemma~\ref{matrix_lemma},
  we know that $\cf(\gamma) = \kappa$ and $\alpha = C_\gamma(\omega\xi)$. It follows that $\alpha \in \Gamma$
  and hence, by Clause~\eqref{cb25} of Lemma~\ref{matrix_lemma}, we have $\xi = \eta_\alpha$.
  Since $|B^\beta_\xi|<\kappa$, and by Clause~\eqref{cb5} of Lemma~\ref{matrix_lemma}, we know that $\gamma \in B^\beta_{<\xi}$,
  so, by the previous paragraph, $(\gamma, \xi)$ is active. Hence, by Definition~\ref{active_defn},
  $(\alpha, \xi)$ is active as well.

  $\br$ If $\xi = \eta_\beta$ and $\gamma \in E^{\kappa^+}_\kappa$ is such that $\beta \in \acc(C_\gamma)$
  and $(\gamma, \xi)$ is active, then by Clauses~\eqref{cb3} and \eqref{cbinfty} of Lemma~\ref{matrix_lemma}, $B^\gamma_\xi \cap (\beta + 1) = B^\beta_\xi$, so $\alpha \in B^\gamma_\xi$.
  Moreover, $\xi > \eta_\gamma = 0$, so, by the previous cases, we again conclude that $(\alpha, \xi)$ is active.
\end{proof}

Our final recursion hypotheses concern non-active pairs $(\beta, \xi)$.

First, suppose that $(\beta, \xi)$ is not active and $\xi = \eta_\beta$. If $\xi \in \acc(\kappa)$ and there is $\gamma \in E^{\kappa^+}_\kappa$ such that
$\beta \in \acc(C_\gamma)$ and $\sup\{\eta < \xi \mid (\gamma, \eta) \text{ is active}\}= \xi$, then we will require that $p^\beta_\xi \in \bb{Q}$ and $x_{p^\beta_\xi} = B^\beta_\xi$.

Next, suppose that $(\beta, \xi)$ is not active and $\xi > \eta_\beta$. Let
\[
  \eta^* := \max\{\sup\{\eta < \xi \mid (\beta, \eta) \text{ is active}\}, \eta_\beta\}.
\]

$\br$ If $\eta^* = \xi$, then we will require that $p^\beta_\xi \in \bb{Q}$ and $x_{p^\beta_\xi} = B^\beta_\xi$.

$\br$ If $\eta^* < \xi$ and $\beta \in x_{p^\beta_{\eta^*}}$, then we will have $p^\beta_{\eta^*} \in \bb{Q}$ and will require that $p^\beta_\xi$ is the
$\leq_{\bb{P}}$-greatest condition $q$ such that $q \leq_{\bb{P}} p^\beta_{\eta^*}$ and, for all $\alpha \in B^\beta_\xi \cap \beta$,
$q \rest (\alpha + 1) = p^\alpha_\xi$.

\subsection{The construction}
We now turn to the actual construction.
Suppose that $\beta < \kappa^+$, $\xi < \kappa$, and we have already
constructed $\langle p^\alpha_\eta \mid \alpha < \kappa^+, ~ \eta < \xi \rangle$ and $\langle p^\alpha_\xi
\mid \alpha < \beta \rangle$. We now construct $p^\beta_\xi$. There are a number of cases to consider.
In all cases, unless explicitly verified, it will be trivial to check that the recursion hypotheses are maintained.
\begin{description}
  \item[Case 0] \textbf{$\xi < \eta_\beta$.} Let $p^\beta_\eta := \one$.
  \item[Case 1] \textbf{$\xi = \eta_\beta$.} There are now a few subcases to consider.
    \begin{description}
      \item[Subcase 1a] \textbf{$(\beta, \xi)$ is active.} In particular, $x_{s_\xi}=\vartheta_\xi^\xi\ge\theta^\beta_\xi$. Let $p^\beta_\xi$ be the unique condition
        $q$ such that $x_q = B^\beta_\xi$ and $\pi^\beta_\xi.q = s_\xi \rest \theta^\beta_\xi$,
        i.e., $p^\beta_\xi = (\pi^\beta_\xi)^{-1}.(s_\xi \rest \theta^\beta_\xi)$.
        Note that, for all $\alpha \in B^\beta_\xi$, Lemma \ref{active_lemma} implies that $(\alpha, \xi)$
        is active. We therefore have $\pi^\beta_\xi.p^\alpha_\xi = \pi^\alpha_\xi.p^\alpha_\xi =
        s_\xi \rest \theta^\beta_\xi$, so $p^\beta_\xi \rest (\alpha + 1) = p^\alpha_\xi$
        and requirement~\eqref{cc3} is satisfied.
      \item[Subcase 1b] \textbf{$(\beta, \xi)$ is not active and there is $\gamma \in E^{\kappa^+}_\kappa$
        such that $\beta \in \acc(C_\gamma)$ and $\sup\{\eta < \xi \mid (\gamma, \eta) \text{ is active}\} = \xi$.}
        Fix such a $\gamma$. Note that $B^\gamma_{<\xi} \cap \beta$ is unbounded in $B^\beta_\xi \cap \beta$ and, for all
        $\alpha \in B^\gamma_{<\xi} \cap \beta$, $\sup\{\eta < \xi \mid (\alpha, \eta) \text{ is active}\} = \xi$.
        Therefore, by our recursion hypotheses, for all $\alpha \in B^\gamma_{<\xi}$, we know that
        $p^\alpha_\xi \in \bb{Q}$ and $x_{p^\alpha_\xi} = B^\alpha_\xi$. By Clause~\eqref{c12} of Definition~\ref{class_def}, there is a unique condition $q \in \bb{Q}$ such that $x_q = B^\beta_\xi \cap \beta$
        and, for all $\alpha \in B^\beta_\xi$, we have $q \rest (\alpha + 1) = p^\alpha_\xi$.
        By Clause~\eqref{c10} of Definition~\ref{class_def}, there is a lower bound $p$ for
        $\langle p^\gamma_\eta \mid \eta < \xi \rangle$ such that:
        \begin{itemize}
          \item $p \in \bb{Q}$;
          \item $x_p = B^\gamma_\xi = B^\beta_\xi \cup \{\gamma\}$;
          \item $p \rest \beta = q$.
        \end{itemize}
        Fix such a lower bound $p$ with a $\lhd_\kappa$-minimal possible value for $\pi_\xi^\gamma.p$,
        and let $p^\beta_\xi := p \rest (\beta + 1)$. Note that, by requirement \eqref{cc5},
        the construction in this Subcase is independent of our
        choice of $\gamma$.
      \item[Subcase 1c] \textbf{Otherwise.} Let $p^\beta_\xi$ be the unique condition, given by Clause~\eqref{c12}
        of Definition~\ref{class_def}, such that $x_{p^\beta_\xi} = \bigcup\{x_{p^\alpha_\xi}\mid \alpha \in (B^\beta_\xi \cap \beta)\}$ and,
        for all $\alpha \in (B^\beta_\xi \cap \beta)$, we have $p^\beta_\xi \rest
        (\alpha + 1) = p^\alpha_\xi$.
    \end{description}
  \item[Case 2] \textbf{$\xi > \eta_\beta$.} There are again a few subcases to consider.
    \begin{description}
      \item[Subcase 2a] \textbf{$(\beta, \xi)$ is active.} Let $p^\beta_\xi$ be the unique condition $q$
        such that $x_q = B^\beta_\xi$ and $\pi^\beta_\xi.q = s_\xi \rest \theta^\beta_\xi$. By
        Definition~\ref{active_defn}, we have that $s_\xi \leq_{\bb{P}} \pi^\beta_\xi.p^\beta_\eta$ for all
        $\eta < \xi$, which implies that $p^\beta_\xi \leq_{\bb{P}} p^\beta_\eta$ for all $\eta < \xi$, so
        requirement~\eqref{cc1} holds.
      \item[Subcase 2b] \textbf{$(\beta, \xi)$ is not active, $\sup\{\eta < \xi \mid (\beta, \eta) \text{ is active}\} = \xi$,
        and $\beta \not\in E^{\kappa^+}_\kappa$.} In this Subcase, we have that $\xi \in \acc(\kappa)$ and
        $B^\beta_{<\xi} \cap \beta$ is unbounded in $B^\beta_\xi \cap \beta$. Since, for all
        $\alpha \in B^\beta_{<\xi}$, we know that $\sup\{\eta < \xi \mid (\alpha, \eta) \text{ is active}\} = \xi$,
        it follows as in Subcase 1b that there is a unique condition $q \in \bb{Q}$ such that
        $x_q = B^\beta_\xi \cap \beta$ and, for all $\alpha \in B^\beta_\xi$, we have $q \rest (\alpha + 1)
        = p^\alpha_\xi$. By Clause~\eqref{c10} of Definition~\ref{class_def}, there is $p \in \bb{Q}$ such that:
        \begin{itemize}
          \item $p$ is a lower bound for $\langle p^\beta_\eta \mid \eta < \xi \rangle$;
          \item $p \in \bb{Q}$;
          \item $x_p = B^\beta_\xi$;
          \item $p \rest \beta = q$.
        \end{itemize}
        Let $p^\beta_\xi$ be such a $p$.
      \item[Subcase 2c] \textbf{$(\beta, \xi)$ is not active, $\sup\{\eta < \xi \mid (\beta, \eta) \text{ is active}\} = \xi$,
        and $\beta \in E^{\kappa^+}_\kappa$.} Let $\alpha := C_\beta(\omega \xi)$, so that $B^\beta_\xi = B^\alpha_\xi
        \cup \{\beta\}$. When defining $p^\alpha_\xi$, we were in Subcase 1b. In that Subcase, we considered
        a $\gamma \in E^{\kappa^+}_\kappa$ such that $\alpha \in \acc(C_\gamma)$, produced a condition
        $p$ with $x_p = B^\gamma_\xi$, and let $p^\alpha_\xi := p \rest (\alpha + 1)$. Let $\pi:B^\gamma_\xi
        \rightarrow B^\beta_\xi$ be the unique order-preserving bijection, and let $p^\beta_\xi = \pi.p$.
        Since, by requirement \eqref{cc5}, we have $\langle \pi^\beta_\xi.p^\beta_\eta \mid \eta < \xi \rangle =
        \langle \pi^\gamma_\xi.p^\gamma_\eta \mid \eta < \xi \rangle$, and since $\pi \restriction B^\alpha_\xi$
        is the identity, the recursion hypotheses are all easily verified.
      \item[Subcase 2d] \textbf{$(\beta, \xi)$ is not active and there is no $\eta < \xi$
        such that $\beta \in x_{p^\beta_\eta}$.} Let $p^\beta_\xi$ be the unique condition $q$
        such that $x_q = \bigcup\{x_{p^\alpha_\xi}\mid \alpha \in (B^\beta_\xi \cap \beta)\}$ and,
        for all $\alpha \in (B^\beta_\xi \cap \beta)$, we have $q \rest (\alpha + 1) = p^\alpha_\xi$.
      \item[Subcase 2e] \textbf{Otherwise.} Let
        \[
          \eta^* := \max\{\sup\{\eta < \xi \mid (\beta, \eta) \text{ is active}\}, \eta_\beta\}.
        \]
        Since we are not in any of the
        previous Subcases, it must be the case that $\eta^* < \eta$, $p^\beta_{\eta^*} \in \bb{Q}$, and
        $x_{p^\beta_{\eta^*}} = B^\beta_{\eta^*}$. For all $\eta \in (\eta^*, \xi]$, let $q_\eta$
        be the unique condition, given by Clause~\eqref{c12} of Definition~\ref{class_def}, such that
        $x_{q_\eta} = \bigcup\{x_{p^\alpha_\eta}\mid \alpha \in (B^\beta_\eta \cap \beta)\}$ and,
        for all $\alpha \in (B^\beta_\eta \cap \beta)$, we have $q \rest (\alpha + 1) = p^\alpha_\eta$.
        By the recursion hypotheses, we know that, for all $\eta \in (\eta^*, \xi)$, $p^\beta_\eta$ is the
        $\leq_{\bb{P}}$-greatest lower bound of $p^\beta_{\eta^*}$ and $q_\eta$, as given by Clause~\eqref{c11}
        of Definition \ref{class_def}. Therefore, if we let $p^\beta_\xi$ be the $\leq_{\bb{P}}$-greatest lower
        bound of $p^\beta_{\eta^*}$ and $q_\xi$, which again exists by Clause~\eqref{c11} of Definition~\ref{class_def},
        it will follow that $p^\beta_\xi \leq_{\bb{P}} p^\beta_\eta$
        for all $\eta < \xi$, so requirement~\eqref{cc1} holds. The other requirements are easily verified.
    \end{description}
\end{description}

This completes the construction. We have maintained requirements \eqref{cc1}--\eqref{cc3} and \eqref{cc5}  throughout. We now verify requirement \eqref{cc4}.
To this end, fix $\beta < \kappa^+$, $i < \kappa$, and  $x \in {\beta + 1\choose \theta_{\mathcal{D}_i}}$. We will find
$\xi < \kappa$ and $q \in D_{i, x}$ such that $p^\beta_\xi \leq_\bb{P} q$.

Fix $j < \kappa$ such that $x \subseteq B^\beta_j$, and fix $\eta^* < \kappa$ such that $(i_{\eta^*}, j_{\eta^*}, z_{\eta^*}) =
(i,j,\pi^\beta_j``x)$. Find $\xi \in \acc(\kappa \setminus (\max\{j,\eta^*, \eta_\beta\} + 1))$ such that:
\begin{itemize}
  \item $\langle \vartheta^\xi_\eta \mid \eta \leq \xi \rangle = \langle \theta^\beta_\eta \mid \eta \leq \xi \rangle$;
  \item $\langle \varpi^\xi_{\eta, \eta'} \mid \eta < \eta' \leq \xi \rangle = \langle \pi^\beta_{\eta, \eta'} \mid \eta < \eta' \leq \xi \rangle$;
  \item $\langle q^\xi_\eta \mid \eta < \xi \rangle = \langle \pi^\beta_\eta.p^\beta_\eta \mid \eta < \xi \rangle$.
\end{itemize}
The following two claims now suffice for the verification of requirement \eqref{cc4}.

\begin{claim}
  $(\beta, \xi)$ is active.
\end{claim}

\begin{proof}
  We verify the requirements in Definition~\ref{active_defn}. We clearly have $\theta^\beta_\xi \leq \vartheta_\xi^\xi$ and $\xi > \eta_\beta$.
  Moreover, for all $\eta < \xi$, we have
  \[
    \varpi^\xi_{\eta, \xi}.q^\xi_\eta = \pi^\beta_{\eta, \xi}.\pi^\beta_\eta.p^\beta_\eta = \pi^\beta_\xi.p^\beta_\eta.
  \]
  Since $p^\beta_\xi \in \bb{P}_{B^\beta_\xi}$ is a lower bound for $\langle p^\beta_\eta \mid \eta < \xi \rangle$,
  it follows that $\pi^\beta_\xi.p^\beta_\xi \in \bb{P}_{\vartheta^\xi_\xi}$ is a lower bound for
  $\langle \varpi^\xi_{\eta, \xi}.q^\xi_\eta \mid \eta < \xi \rangle$. In particular, $\xi \in X$.
  It follows that $s_\xi$ is a lower bound for $\langle \varpi^\xi_{\eta, \xi}.q^\xi_\eta
  \mid \eta < \xi \rangle = \langle \pi^\beta_\xi.p^\beta_\eta \mid \eta<\xi\rangle$, which completes the verification.
\end{proof}

\begin{claim}
  There is $q \in D_{i,x}$ such that $p^\beta_\xi \leq_\bb{P} q$.
\end{claim}

\begin{proof}
  Since $(\beta, \xi)$ is active and $\theta^\beta_\xi = \vartheta^\xi_\xi$, we have $\pi^\beta_\xi.p^\beta_\xi = s_\xi$.
  It thus suffices to find $q' \in D_{i,\pi^\beta_\xi``x}$ such that $s_\xi \leq_\bb{P} q'$.

  Note that $\eta^*, j < \xi$ and $\pi^\beta_j``x \subseteq \theta^\beta_j = \vartheta^\xi_j$. Therefore,
  since $(i,j,\pi^\beta_j``x) = (i_{\eta^*}, j_{\eta^*}, z_{\eta^*})$ and $\langle s_\xi \mid \xi \in X \rangle$
  satisfies the conclusion of Lemma~\ref{s_xi_proposition}, it follows that there is $q' \in D_{i, \varpi^\xi_{j,\xi}``
  \pi^\beta_j``x} =  D_{i, \pi^\beta_\xi``x}$ such that $s_\xi \leq_\bb{P} q'$, as desired.
\end{proof}

\section*{Acknowledgments}
In 2010, a few days after attending his talk at the \emph{11th International Workshop on Set Theory in Luminy}, Foreman wrote to the second author that one
can construct an $\aleph_2$-Souslin tree from the conjunction of $\square_{\aleph_1}$ and
$\diamondsuit(\aleph_1)$, using ideas from Shelah's ``models with second order properties'' papers and \cite{foreman_dense_ideal}.
This work was never published, and no details of the construction were provided,
but this hint turned out to be quite stimulating. We thank him for pointing us in this direction.

Portions of this work were presented by the first author at the \emph{Oberseminar mathematische Logik} at the University of
Bonn in May 2017 and at the \emph{6th European Set Theory Conference} in Budapest in July 2017.
We thank the organizers for their hospitality.

We also thank the referee for their thoughtful feedback.

\end{document}